\documentclass[11pt]{amsart}

\usepackage{amscd,amssymb,amsmath,latexsym,enumerate}
\usepackage[utf8]{inputenc}
\usepackage{amscd,amssymb,amsmath,amsthm,bbm}
\usepackage[mathscr]{euscript}
\usepackage{mathrsfs}

\usepackage[breaklinks=true,colorlinks=true,
linkcolor=black,urlcolor=black,citecolor=black,
bookmarks=true,bookmarksopenlevel=2]{hyperref}

\DeclareMathAlphabet{\mathpzc}{OT1}{pzc}{m}{it}

\newtheorem{definition}{Definition}[section]
\newtheorem{proposition}[definition]{Proposition}
\newtheorem{lemma}[definition]{Lemma}
\newtheorem{theo}[definition]{Theorem}

\newtheorem{corollary}[definition]{Corollary}
\newtheorem{remark}[definition]{Remark}

\newtheorem{cor}{Corollary}

\newtheorem{theorem}{Theorem}

\newtheorem{theorema}{Theorem}[section]

\theoremstyle{remark}

\renewenvironment{proof}{{\bfseries Proof.}}{\hfill$\Box$}

\newcommand{\CC}{{\mathbb C}}
\newcommand{\NN}{{\mathbb N}}

\newcommand{\RR}{{\mathbb R}}
\newcommand{\ZZ}{{\mathbb Z}}

\newcommand{\Cc}{{\mathcal C}}
\newcommand{\Hh}{{\mathcal H}}
\newcommand{\Ll}{{\mathcal L}}
\newcommand{\Mm}{{\mathcal M}}
\newcommand{\Nn}{{\mathcal N}}
\newcommand{\Pp}{{\mathcal P}}
\newcommand{\Rr}{{\mathcal R}}
\newcommand{\Tt}{{\mathcal T}}
\newcommand{\Ff}{{\mathcal F}}
\newcommand{\Oo}{{\mathcal O}}

\newcommand{\Xx}{{\mathcal X}}

\newcommand{\cs}{{\mathscr C}}
\newcommand{\gs}{{\mathscr G}}
\newcommand{\is}{{\mathscr I}}
\newcommand{\ks}{{\mathscr K}}
\newcommand{\os}{{\mathscr O}}
\newcommand{\us}{{\mathscr U}}
\newcommand{\vs}{{\mathscr V}}

\newcommand{\supp}{{\operatorname{supp}}}			
\newcommand{\Del}{{\mathit Del_{U,K}}}				
\newcommand{\Delr}{{\mathit Del_{*}}}				
\newcommand{\DelrG}{{\mathit Del_{*}^{G}}}			
\newcommand{\DelN}{{\mathit Del_{U,K}^{e}}}			
\newcommand{\inv}{{\mathit \is_{U,K}}}				
\newcommand{\invN}{{\mathit \is_{U,K}^{e}}}			
\newcommand{\invG}{{\mathit \is_{G}(X)}}			
\newcommand{\IDelGH}{{\mathit \is_{G\times H}(\Delr)}}			
\newcommand{\IDelG}{{\mathit \is_{G}(\Delr)}}			
\newcommand{\Lat}{{\mathit Lat}}					
\newcommand{\win}{\mathfrak{v}} 					
\newcommand{\Kfin}{{\mathcal K}^{fin}(G\times H,\win)}	
\newcommand{\KfinG}{{\mathcal K}^{fin}(G)}			
\newcommand{\PE}{{\mathcal P}{\mathcal E}(G,\Rr)}				
\newcommand{\Cfin}{{\mathfrak C}(G\times H,\win)}	

\allowdisplaybreaks

\usepackage{xcolor}

\begin{document}
\title{Delone dynamical systems and spectral convergence}
\author{Siegfried Beckus, Felix Pogorzelski}

\address{Department of Mathematics\\
Technion - Israel Institute of Technology\\
32000 Haifa, Israel}
\email{beckus.siegf@technion.ac.il}

\address{Department of Mathematics, University of Leipzig\\
04109 Leipzig, Germany}
\email{felix.pogorzelski@math.uni-leipzig.de}

\begin{abstract}
In the realm of Delone sets in locally compact, second countable, Hausdorff groups, 
we develop a dynamical systems approach in order to study the continuity behavior of measured quantities arising from point sets.
A special focus is both on the autocorrelation, as
well as on the density of states for random bounded
operators. It is shown that for uniquely ergodic limit systems, the latter measures behave continuously with respect to the Chabauty-Fell convergence of hulls. In the special situation of Euclidean spaces, our results complement recent developments in describing spectra as topological limits: we show that the measured quantities under consideration can be approximated via periodic analogs. 
\end{abstract}


\maketitle

%


\section{Introduction and main theme}

 
The spectral theory of random Schr\"odinger operators in aperiodic media is full of interesting, surprising and challenging phenomena.  During the past decade, tools from  a wide range of mathematical areas have been used and designed in order to gain deep insight into the arising spectral theoretic patterns. 
In this work we develop a dynamical systems approach for the sake of approximating spectral quantities of operators arising in quantum mechanical models of non-periodic solids. 
For this purpose, this work brings together a large variety of techniques from group dynamics, operator algebra theory and spectral theory. While our results apply to 
prominent quantities from mathematical physics, the techniques in the proofs rely on a general convergence principle for group actions. 
Furthermore, the toolbox developed in the present paper is interesting in its own right since it provides the basis for tackling 
exciting questions in dynamical systems theory. For instance, it can be used to construct compact spaces consisting of strictly ergodic translation actions. 

\medskip

With the help of transfer matrices and the trace map formalism, the spectral theory of the one-dimensional case was tremendously pushed forward during the last decades \cite{Sut87,BIST89,Bel90,BeBoGh91,Dam98,DaLe99I,DaLe99II,Len02}. Another remarkable outcome of these methods is the solution of the Ten Martini Problem \cite{AvJi09} by Avila and Jitomirskaya. Astonishing achievements have been also obtained for the Fibonacci Hamiltonian summarized in \cite{DaGoYe16}. While many facets of one-dimensional random Schr\"odinger operators are nowadays understood in quite significant depth, it seems fair to say that the picture is still rather unclear in the higher-dimensional situation. In fact, most attempts aiming at rigorously advancing the known methods to higher dimensions have remained incomplete so far, up to few exceptions \cite{LeSt06,DaGoSo13,DaGo16}.
 
\medskip

One meta strategy to approach higher-dimensional spectral theory in aperiodic geometries is via approximation. Recent results \cite{BeBe16,BeckusThesis16,BeBeNi17,BeBeCo17} started to implement a systematic approximation theory for the spectrum via approximating the underlying structures. Specifically, this reflects in the (H\"older)-continuity of the map
$$
\Sigma:\{\text{subsystems of } X\}\to\{K\subseteq\CC \text{ compact}\},\quad Y\mapsto\sigma(A_Y),
$$
for every normal random operator $A$ over a dynamical system $(X,G)$, where $\sigma(A_Y)$ is the spectrum of the restriction $A_Y$ to the subsystem $Y\subseteq X$, c.f. \@ \cite{BeckusThesis16,BeBeNi17,BeBeCo17}. Here, $X$ is a compact space and $G$ is a countable Hausdorff amenable group that acts continuously on $X$. Furthermore, the set of subsystems (invariant, closed subsets of $X$) is naturally equipped with the Chabauty-Fell topology \cite{Cha50,Fel62}. As usual, the convergence of compact subsets of $\CC$ is measured in the Hausdorff metric. This has been worked out in a much more general setting in \cite{BeBeNi17},
where also the case of Delone sets is covered. As such, these insights open the possibility 
to handle very interesting examples such as the Penrose tiling.

\medskip

Inspired by the mathematical and physical experiences so far, periodic structures serve as the best candidates for approximations since the machinery of Floquet-Bloch is available to analyze such systems. A simple compactness argument shows that the Chabauty-Fell closure of periodic subsystems is non-trivial by sending the periods to infinity. Thus, there always exist non-periodic subsystems that can be approximated in the Chabauty-Fell topology by periodic subsystems. However, the main difficulty lies in the classification of those subsystems. 
In one dimension, this has been done in \cite{BeckusThesis16,BeBeNi17}. Furthermore, for subsystems on $\ZZ^d$ defined by a substitution rule invoking local symmetries, the work \cite{BeckusThesis16} provides sufficient conditions for periodic approximability. In both cases, 
the constructions are explicit and they include both known results such as for the Fibonacci sequence, but also new examples, e.g.\@ for the Table tiling.

\medskip

The previously discussed approximation theory is the starting point of our project. For a Delone set $D$ in a general locally compact group $G$, the associated Delone dynamical system arises naturally as the translation action on the closure of the orbit $G.D$.  
We investigate continuity properties of two spectral quantities based on the convergence of invariant probability measures:
 
\begin{itemize}
\item {\bf Autocorrelation of Delone sets.} \\
The notion of mathematical autocorrelation was developed for describing diffraction experiments in a rigorous manner. Based on investigations on the diffractive behaviour of solids, Dan Shechtman discovered the physical existence of quasicrystalline structures \cite{SBGC84}. For this  groundbreaking observation, Shechtman was awarded the Nobel prize in chemistry in 2011. A suitable approach to model the underlying physical systems is via the class of Delone sets which are uniformly scattered point sets representing the atomic nuclei. It was Meyer \cite{Mey69,Mey70,Mey72} who started the investigation of certain generalizations of lattices ({\em Meyer sets}) in the context of number theoretical questions.
Shortly after Shechtman's experiments, mathematical physicists \cite{KrNe84,LeSt84,Hof95} began to develop a mathematical diffraction theory. It was explored shortly afterwards that many diffractive properties of Delone sets are reflected 
in spectral properties of the corresponding dynamical system. For results of this kind, see e.g.\@ the non-exclusive list of references \cite{Queffelec87,Dw93,Sch00, LeMoSo02,BaMo04,BaLe04,Gou05,LeStr09, BjHaPo17}. In particular situations, dynamical stability assumptions 
carry over to the diffractive level \cite{BaLe05}. 
For a detailed introduction and overview, we refer further to the recent monograph \cite{BaakeGrimm13} and references therein. Recently, a new non-commutative spherical diffraction theory was carried out in \cite{BjHaPo16,BjHaPo17} for regular model sets. These results will allow us to approximate the autocorrelation even in non-abelian settings. Analogous results for diffraction measures are yet to be explored. 

\item {\bf Density of states (DOS) for random bounded operators.}\\
The density of states (DOS) of an operator on a Hilbert space encodes its spectral distribution. For self-adjoint elements, it gives rise to a spectral measure on $\RR$. For random bounded operators arising from a Delone dynamical systems, the DOS is represented by the Pastur-Shubin trace formula \cite{Pas71,Shu79}. In mathematical physics, one is often interested in knowing the cumulative distribution function of the DOS, called {\em integrated density of states (IDS)}. In quantum mechanical models, the latter quantity provides information about the averaged number of states of an electron per unit volume below some fixed energy level.
A mathematically intuitive understanding of the DOS, respectively the IDS for operators on graphs is given by approximations via finite volume analogs as studied in \cite{Bel86,Len02, LeSt05, LMV08,Ele08, KeLe13, LV09, LSV11,Pog14, ScSc15, PS16}. For our purposes, the Pastur-Subin trace formula is the more suitable starting point since here the DOS is represented as an integral over the Delone dynamical system. We make use of the developments achieved in \cite{LeSt03-Delone,LeSt03-Algebras,LePeVe07} which give a description of the DOS in purely dynamical terms.  
\end{itemize}

Our main goal is to prove weak-$*$-continuity of these two spectral quantities with respect to convergence of the underlying dynamical systems. The central viewpoint is to consider classes of Delone dynamical systems as subsystems of one single topological dynamical system. 
As for the DOS, this also allows us to interpret random bounded operators as restrictions to these subsystems. 

\medskip

\textbf{Continuity result:} Let $(\Omega,\mu,G)$ be a uniquely ergodic Delone dynamical system.
Consider a sequence $(\Omega_n,\mu_n,G)$  of Delone dynamical systems 
such that $\Omega_n \to \Omega$ in the Chabauty-Fell topology. 
Then the following assertions hold.
\begin{itemize}
\item[(A)] The autocorrelations $\gamma_{\Omega_n,\mu_n}$ converge to the autocorrelation $\gamma_{\Omega,\mu}$ in the weak-* topology of Radon measures on $G$. 
\item[(B)] For every self-adjoint random bounded operator $A$, the DOS $\eta^A_{\Omega_n,\mu_n}$ converge to the DOS $\eta^A_{\Omega,\mu}$ in the weak-* topology of Radon measures on $\RR$. 
\end{itemize}
The precise statement of~(A) is provided in Corollary~\ref{cor:Appr-AC} based on Theorem~\ref{thm:Appr-wgtAC} and proven in Section~\ref{sec:ConAC}.
Statement~(B) is formulated rigorously in Corollary~\ref{cor:Appr-DOS} based on Theorem~\ref{thm:Appr-wgtDOS} and shown in Section~\ref{sec:RanOpDOS}.
The proof is based on an abstract convergence result of invariant probability measures for dynamical systems, cf.\@ Theorem~\ref{thm:mainapprox} in Section~\ref{sec:main}. One exploits 
the fact that the relevant data is encoded in the measures
$\mu_n, \mu$ defined in a common ``large'' ambient space. This point of view has been considered before for Delone dynamical systems in \cite{LeMo09}. In our situation, the convergence of measures is a consequence of the convergence of the underlying hulls.
 

\medskip

With these continuity results at our disposal, we can settle the approximation
of the autocorrelation and the DOS in the general context of regular model sets. The latter sets
are described by two groups $G$ and $H$, a cocompact lattice $\Gamma$ in the product $G \times H$,
and a compact subset $W \subseteq H$ called the {\em window}. A model set is then obtained
by projecting all lattice points in the cut out strip $G \times W$ to $G$.
Under certain regularity conditions on the lattice and the window, the assumption of unique ergodicity is automatically satisfied. This fact relies on the 
{\em parametrization map} relating the corresponding Delone dynamical system with the 
homogeneous type dynamics arising from the cut-and-project scheme. For abelian groups, this connection
has been established by Schlottmann \cite{Sch00}, and for general locally compact, second countable groups, this has recently been shown in \cite{BjHaPo16,BjHaPo17}. To deal with the natural discontinuities of the cutting,
we approximate the characteristic function $\chi_W$ of the window by continuous functions (called {\em window functions}) with compact support in $H$. This leads to the notion of {\em weighted Delone sets}, where the weight is determined by the window function. Due to the $G$-invariance of these
functions, there is a canonical notion of translation dynamical systems for weighted Delone sets.
For the sake of different applications, weighted model sets have been considered before, see e.g.\@ \cite{Baa01,BaMo04,LeRi07,Str14,RiSt15,RiSt17,BaHuSt17}. More specifically, weighted model sets were used as a tool to analyze weak model sets which are being studied intensively in ongoing research activitiy. For recent results on the topic, we refer to \cite{Moo02,HuRi15,BaHuSt17} and references therein.

\medskip

Regular model sets provide a vast collection of examples used in physics and mathematics. In the Euclidean setting, it is well known that a defect in the spectrum can be created by changing the lattice. This has for instance been observed in the Kohmoto model \cite{OsKi85,BIT91}. Such errors occur as the cutting process is not continuous. 
On the other hand, small changes of the lattice can always provide periodic systems in the Euclidean setting. Combining these two observations, the task is to show convergence of the measured quantities by changing the lattice without creating a defect. 
Via the parametrization map, the dynamical system of the hull of the regular model set is measurably isomorphic to the homogeneous space defined by the lattice. With this at hand and by approximating the characteristic function function of the window by continuous functions, the desired convergence is obtained without creating an error. This is due to the fact that the measured quantities are not sensitive to them as they are created at the boundary of the window. In particular, the defects observed in the Kohmoto model are eigenvalues of finite multiplicity which are not observed in the DOS.
It is worth pointing out that the hull and the quotient space $G \times H/ \Gamma$ are
quite different from a topological perspective, since the hull is totally disconnected. 
However, dealing with
strongly pattern equivariant potentials of the Hamiltonian, the
Schr\"odinger operator can be lifted to the quotient space while keeping all measured information.

\medskip

Our approximation result involves two limit processes. The first is convergence of lattices in the Chabauty-Fell topology, the second is the approximation of $\chi_W$ by continuous window functions. The lattice convergence implies the convergence of the associated homogeneous spaces.
This follows from the more general statement given in Theorem~\ref{thm:TransHullHomeo} which characterizes convergence of Delone dynamical systems.
Using the parametrization map, we lift all relevant quantities defined over the dynamical system of the regular model set to analogs defined on the homogeneous space. From the dynamical point of view, this approach is in spirit of \cite{KeRi16}, where dynamical properties of (weak) model sets are analyzed. 
As for the autocorrelation, respectively for the DOS, we can represent both measures
as an integral over the Haar measure $m_Y$ of $G \times H / \Gamma$, and the properties of $W$ are reflected only in the integrand. This is a key observation for proving approximation:
keeping unique ergodicity, there is a canonical way to define ``smoothed'' approximations in terms of continuous window functions. 
This puts us in the convenient situation to apply our continuity results from above to obtain convergence as $\Gamma_n \to \Gamma$ simultaneously for all approximating window functions. 
In the full generality of random bounded operators, it is impossible to lift the integrands for the DOS to functions on $G \times H / \Gamma$. However, we show that this can be done for the large class of strongly pattern strongly equivariant Schr\"odinger operators. Here, the potentials are given by strongly pattern equivariant functions which have been introduced in \cite{KePu00,Kel03,Kel08}. 

\medskip

\textbf{Approximations of regular model sets:} Let $P=P(G,H,\Gamma,W)$ be a regular model set. Consider a sequence of cocompact lattices $(\Gamma_n)$ converging to the cocompact lattice $\Gamma$ and a sequence of continuous window functions $(\win_l)$ converging pointwise monotonically to $\chi_W$. Then the following assertions hold.

\begin{itemize}
\item[(C)] The autocorrelations $\gamma_{\Gamma_n,\win_l}$ converge to the autocorrelation $\gamma_P$ in the weak-* topology of Radon measures on $G$. 
 
\item[(D)] For every strongly pattern equivariant Schr\"odinger operator $A$, the density of states $\eta^A_{\Gamma_n,\win_l}$ converge to the density of states $\eta^A_{P}$ in the weak-* topology
of Radon measures on $\RR$.  
\end{itemize}

In accordance with the above description, one first takes the limit in $n$, and afterwards in $l$.
The statement~(C) is formulated precisely in Theorem~\ref{thm:Appr-ModAC} and proven in 
Subsection~\ref{ssec:Appr-CP-AC}. The rigorous assertion of statement~(D) is given in 
Theorem~\ref{thm:Appr-ModDOS} and shown in Subsection~\ref{ssec:Appr-DOS-PEop}.

\medskip

We emphasize at this point that for the Euclidean cut-and-project schemes, one obtains 
approximations via quantities defined on periodic structures. To see this, note that lattices in 
$\RR^{d}$ are of the form $\Gamma = S \ZZ^d$ for some $S \in \operatorname{GL}_d(\RR)$. Hence, approximations of $\Gamma$ can be given in the form $\Gamma_n = Q_n\ZZ^d$, where all entries of
$Q_n$ are rational and $Q_n \to S$.  

\medskip

\textbf{Outlook}. 
The results of this paper are based on the fact that the dynamical systems
associated to the limit objects are uniquely ergodic. For regular model sets, this is the case. However, the dynamical convergence is induced by the convergence of lattices, whereas the approximation results of Theorem~\ref{thm:Appr-wgtAC} and~\ref{thm:Appr-wgtDOS} are valid 
for much more general Delone dynamical systems. This naturally raises the question of sufficient criteria on such systems to be uniquely ergodic. 
In view of new developments on Delone sets in quite general spaces, it is natural to address the question whether analog results remain
true in the setting of certain non-abelian groups such as the Heisenberg group.
In the Euclidean situation, Lagarias and Pleasants \cite{LaPl03} have shown that linear repetitive or densely repetitive Delone sets give rise to uniquely ergodic
hulls. By imposing quantitative restrictions on the repetivity, one obtains compact spaces of uniquely ergodic Delone dynamical systems. 
This and more will be addressed in a subsequent paper. 

\medskip

Imposing suitable amenability conditions, our continuity results even imply convergence of the spectra of the random bounded operators, cf.\@ Remark~\ref{rem:convergenceofspectra}.
This raises naturally the question if this remains true in the case of model sets as discussed in Theorem~\ref{thm:Appr-ModDOS}. The main task will be to 
show that the approximation of the window functions is compatible with the convergence of the spectra. 


\medskip

\textbf{Organization of the paper}. 
In Section~\ref{sec:main}, the basic concepts of dynamical systems and Chabauty-Fell convergence
of subsystems are introduced. Furthermore, we prove in Theorem~\ref{thm:mainapprox} a general 
convergence principle of invariant probability measures via topological convergence of dynamical systems.
Delone sets and Delone dynamical systems and their transversals are introduced in Section~\ref{sec:DDS}. We first characterize the convergence of Delone dynamical systems by the convergence of the associated transversals, cf.\@ Theorem~\ref{thm:TransHullHomeo}. Furthermore, compact spaces of lattices are analyzed in Theorem~\ref{thm:LatCom}. Section~\ref{sec:Appr-AC} is devoted to weighted Delone sets and the convergence of the autocorrelation, cf.\@ Theorem~\ref{thm:Appr-wgtAC} and Corollary~\ref{cor:Appr-AC}. Next, we discuss random bounded operators and the density of states in Section~\ref{sec:RanOpDOS}. The convergence of the DOS is settled in Theorem~\ref{thm:Appr-wgtDOS} and Corollary~\ref{cor:Appr-DOS}. The previous continuity results are applied to regular model sets in Section~\ref{sec:CP-schemes}. Here, Theorem~\ref{thm:Appr-ModAC} covers the autocorrelation, and Theorem~\ref{thm:Appr-ModDOS} treats the DOS for strongly pattern equivariant Schr\"odinger operators. We give an appendix in Section~\ref{sec:appendix} about groupoid von Neumann algebras and the Pastur-Shubin trace formula.

\medskip

\textbf{Acknowledgments}.
We thank the organizers and the participants of the Oberwolfach Workshop 1740 ``Spectral Structures and Topological Methods in Mathematical Quasicrystals'' for various enlightening discussions. Special thanks go to Daniel Lenz for drawing our attention to relevant references and for raising interesting follow-up questions.
FP is thankful for financial support through a Technion Fine Fellowship.


\section{A general approximation theorem} \label{sec:main}

This section provides an abstract convergence principle for dynamical systems. The precise assertion is stated and proven in Subsection~\ref{ssec:ConPrin}. In the following, we clarify the setting and introduce basic notions of dynamical systems.

\subsection{Preliminaries} \label{ssec:prelim}

Throughout the whole paper, $G$ denotes lcsc group with neutral element $e$. When writing lcsc, we will always refer to a locally compact, second countable group which is additionally Hausdorff. 
The group $G$ comes along with a left-invariant Radon measure, called the {\em left Haar measure}, which is unique up to constants. Furthermore, we suppose that $G$ is unimodular, i.e.\@ the unique left-invariant Haar measure is also invariant from the right. Hence, we will just use the term {\em Haar measure}. The restriction to unimodular groups is not necessary for all our results. However, in our main applications, we have to deal with the existence of lattice subgroups and of invariant measures on groupoids induced by measure preserving dynamical systems. The existence of either of those objects gives a necessary criterion on the group to be
unimodular.
We will further assume that $G$ as introduced before act by homeomorphisms on compact, metrizable spaces $X$. This
way, we obtain a {\em topological dynamical system}, denoted by $(X,G)$.
In most situations, we will additionally assume that there is a probability
measure $\mu$ on $X$ which is invariant under the action of $G$. 
This means that for every Borel set $A \subseteq X$ and each $g \in G$,
we have $\mu(gA) = \mu(A)$.  The triple
$(X,\mu,G)$ is then referred to as a {\em measure dynamical system}. In the 
sequel, we will just use the term {\em dynamical system}, since the focus on
topological or measure aspects shall become clear from the context. For
a dynamical system $(X,G)$, we say that $Y$ is an invariant set if $gY \subseteq Y$
for all $g \in G$. A measure dynamical system $(X,\mu,G)$ is said to be {\em ergodic}
if for all invariant sets $Y \subseteq X$, one has $\mu(Y) \in \{0,1\}$. 
Note that for non-amenable groups, 
the existence of invariant measures is not guaranteed. However, in all applications
we have in mind, one can find them and in many situations, there is even
only one such measure. In this context, we say that the dynamical systems $(X,G)$ 
and $(X,\mu,G)$
are {\em uniquely ergodic} if $\mu$ is the only $G$-invariant probability measure
on $X$. It is easy to see that then, $\mu$ is in fact ergodic.
In the presence of a dynamical system $(X,G)$ and a compact invariant set $Y\subseteq X$,
we call the dynamical system $(Y,G)$ a {\em subsystem} of $(X,G)$. In non-ergodic
situations, there might be non-trivial subsystems $(Y,G)$ 
carrying a $G$-invariant probability measure $\mu$. This is the central viewpoint for
our approach: we will deal with (uniquely) ergodic subsystems $(Y,\mu,G)$, which we will also
interpret as non-ergodic dynamical systems $(X, \tilde{\mu},G)$ over 
the ambient space, where $\tilde{\mu}$ is the natural extension of $\mu$ to $X$.

\medskip

The space of Radon measures $\Rr(X)$ on the Borel $\sigma$-algebra of the compact Hausdorff space $X$ is identified with $\Cc_c(X)'$ of continuous linear functionals on the space $\Cc_c(X)$. Then $\Rr(X)$ is equipped with the induced weak-$\ast$ topology. If a group $G$ acts on $X$, the subspace of $G$-invariant, probability measures is denoted by $\Mm(X)$. Recall that the support of $\mu\in\Mm(X)$ is defined by
$$
\supp(\mu) \;
	:= \; \big\{
			x\in X
			\,\big|\,
			\text{for all open } V\ni x \Rightarrow \mu(V)>0
		\big\}
		\,.
$$
Note that $\supp(\mu) \in \ks(X)$ for every $\mu \in \Mm(X)$. 

\medskip

Let $Z$ be a locally compact, second countable, Hausdorff space. Then $Z$ is a {\em normal} space (which means that disjoint closed subsets can be separated by open
sets) and also metrizable. Denote by $\ks(Z)$ the set of all compact subsets of $Z$,
and by $\cs(Z)$ the collection of all closed subsets of $Z$. For suitable choices
of $Z$, we will endow the 
spaces $\ks(Z)$ and $\cs(Z)$ with natural topologies.
Precisely, we will use the so-called {\em Chabauty-Fell topology}. It was 
Chabauty~\cite{Cha50} who topologized the spaces of subgroups in countable groups.
Some years later, Fell \cite{Fel62} defined a generalization of the Chabauty-topology on 
$\ks(Z)$ for general topological spaces $Z$.
As a side remark, we point out that other topologies on the space of closed
or compact subsets of topological spaces have been invented and investigated
much earlier, e.g. the {\em Hausdorff topology} \cite{Hausdorff14} 
over metric spaces or the
{\em Vietoris topology} \cite{Vie22} dealing 
with closed rather than with compact sets for the parametrization of the basis. 
A basis for the Chabauty-Fell topology on $\cs(Z)$ is given by 
$$
\vs(K,\os)\; 
	:= \; \big\{Y\in\cs(Z)\;|\; 
		K\cap Y=\emptyset\,,\; 
		O\cap Y\neq\emptyset \text{ for all } O\in\os
	\big\}
$$
where $K\subseteq Z$ is a compact subset and $\os$ is a finite family of open subsets of $Z$.
In an analogous manner, one defines the Chabauty-Fell topology on $\ks(Z)$.

\medskip

The corresponding Chabauty-Fell topology on $\cs(Z)$ is metrizable and complete, and 
the space $\cs(Z)$ is compact with respect to this topology, see \cite{Fel62,Bee93}.
Using the description of the basis given above, we obtain a canonical concept of
convergence. 
Namely, $Y_n \to Y$ in $\cs(Z)$ if for all neighborhoods $\mathcal{U}$ of
$Y$, there is some $N \in \NN$ such that $Y_n \in \mathcal{U}$
for each $n \geq N$.

\subsection{Convergence principle} \label{ssec:ConPrin}

This subsection is devoted to present a convergence result on dynamical systems.

\begin{theorem}\label{thm:mainapprox}
Let $G$ be a lcsc group which acts by homeomorphisms on a compact, second countable, Hausdorff space $X$.
Assume that $(Y_l, \mu_l,G)$ $(l \in \NN)$ are subsystems of $(X,G)$ and suppose that 
$(Y,\mu,G)$ is a uniquely ergodic subsystem of $(X,G)$. \\
If $Y_l \to Y$ in the Chabauty-Fell topology on $\ks(X)$, then $\mu_l \to \mu$ in the weak-* topology of $\Mm(X)$.
\end{theorem}

The theorem gives a continuity result for suitably topologized dynamical systems. This viewpoint is underscored by the fact that every subsystem $(Y,G)$ can be interpreted as a point in $\invG$, where the latter space denotes the space of all $G$-invariant, closed (in fact compact)
subsets of $X$. Note that $\invG$ is a closed subspace of $\cs(X)$
with respect to the Chabauty-Fell topology, c.f. \cite{BeBeNi17} and \cite[Proposition~3.2.5]{BeckusThesis16}.

\medskip

The proof of Theorem~\ref{thm:mainapprox} builds on the following key lemma.

\begin{lemma}
\label{lem-keyMeas}
Let $X$ be a compact, second countable Hausdorff space. Consider a sequence of probability measures $\mu_l\in\Mm(X),\, l\in\NN,$
with their supports $\supp(\mu_l)$ being contained in $C_l \in\ks(X),\, l\in\NN$. Suppose that $C_l$ converges 
to $C \in \ks(X)$ in the Chabauty-Fell topology. \\
Then, the support of every weak-* accumulation point of $(\mu_l)_l$ must be contained in $C$.
\end{lemma}

\begin{proof}
Without loss of generality, assume that $\mu_l\rightharpoonup\mu$ (otherwise pass to a subsequence). Assume $\supp(\mu)\not\subseteq C$ holds. Thus, there is an $x\in \supp(\mu)\setminus C$. Since $X$ is a normal space, Urysohn's Lemma applies. Specifically, there exists an $f\in\Cc_c(X)$ satisfying $0\leq f\leq 1$, $f(x)= 1$ and $\supp(f)\cap C=\emptyset$. Consider the compact set $K:=\{y\in X\,|\, f(y)\geq \frac{1}{2}\}\subseteq\supp(f)$. Then a short computation yields
$$
\mu(f) \;
	= \; \int_X f\, d\mu \;
	\geq \; \int_K f \, d\mu \;
	\geq \; \frac{1}{2}\cdot \mu(K)
	\,.
$$
Since $K$ is a neighborhood of $x\in\supp(\mu)$, $\mu(K)>0$ follows. The weak-* convergence of the measures $(\mu_l)_l$ to $\mu$ implies $\lim_{l\to\infty} \mu_l(f)=\mu(f)$. On the other hand, the Chabauty-Fell open set $\us(\supp(f),\{X\})$ is by our assumption a neighborhood of $C$. By the convergence of $C_l$ to $C$, there is an $l_0\in\NN$ such that $C_l\in\us(\supp(f),\{X\})$ for $l\geq l_0$. Since $\supp (\mu_l) \subseteq C_l$, it is implied that $\supp(f)\cap\supp(\mu_l)=\emptyset$ for $l\geq l_0$. Consequently, $\mu_l(f)=0$ follows for $l\geq l_0$ contradicting $\lim_{l\to\infty} \mu_l(f)=\mu(f)>0$.
\end{proof}

\medskip

We are now ready to prove the abstract convergence theorem of the present paper.
We show that for uniquely ergodic subsystems $(Y,\mu,G)$, the measure $\mu$
is obtained as a weak*-limit of the measures $\mu_l$ associated with 
subsystems $(Y_l, G)$,
whenever $Y_l$ converges to $Y$ in the Chabauty-Fell topology.  

\medskip

{\bf Proof of Theorem~\ref{thm:mainapprox}.}
Since the space $\Mm(X)$ is compact in the weak-* topology, and since we can interpret the measures $\mu_l$
as elements in $\Mm(X)$, we can use the theorem of Banach-Alaoglu to obtain a weak-* accumulation point 
$\widetilde{\mu}\in \Mm(X)$. The support of the measures $\mu_l$ is clearly contained in $Y_l$. Thus,
we deduce from Lemma~\ref{lem-keyMeas} that $\supp(\widetilde{\mu}) \subseteq Y$.
Hence, $\widetilde{\mu}$ is a $G$-invariant probability measure on $Y$. However, the dynamical system
$(Y,G)$ is uniquely ergodic by assumption, and so $\widetilde{\mu} = \mu$. \hfill$\Box$


\section{Delone dynamical systems} \label{sec:DDS}

In the upcoming section, we introduce Delone dynamical systems over a lcsc group $G$. As a special case, we study the subspace of lattices. Delone sets are discrete point sets in $G$ that are uniformly discrete with bounded gaps. The discreteness is measured by an open unit neighborhood $U\subseteq G$ and the maximal size of the gaps is determined by a compact neighborhood $K\subseteq G$ of the unit. The precise definition of $(U,K)$-Delone sets is given in Definition~\ref{def:LSepRDen}. The group $G$ acts naturally by translation on $\Del$ and we analyze the space $\inv:={\mathit \is_{G}\big(\Del\big)}$. For $\Omega\in\inv$, the subset $\Tt\subseteq\Omega$ of all elements $D\in\Omega$ such that $e\in D$ is called transversal $\invN$ is the space of transversals equipped with the induced Chabauty-Fell topology. For the issue of topological spectral approximation, it is convenient in many situations to consider convergence properties of transversals \cite{BeBeNi17}. However, for the approximation of measured spectral quantities in Delone dynamical systems, the suitable approach is to investigate convergence of hulls in the Chabauty-Fell topology. The reason for this is that the hull is invariant with respect to the $G$-action, which allows us to work with $G$-invariant measures and to apply the results from Section~\ref{sec:main}. This cannot be expected for the transversal. The following Theorem~\ref{thm:TransHullHomeo} shows that topologically, convergence of hulls and convergence of transversals is the same. As usual, for $\Tt \in \ks(\Del)$ and a subset $A \subset G$, we define $A \cdot \Tt:= \bigcup_{a \in A} a\cdot \Tt$. 

\begin{theorema}
\label{thm:TransHullHomeo}
The map $\Phi:\invN\to\inv,\, \Tt \mapsto G\cdot \Tt,$ is a homeomorphism. In particular, for $D_n,D\in\Del$, the transversals
$\Tt^{D_n}$ converge to $\Tt^D$ if and only if the hulls $\Hh^{D_n}$ converge to $\Hh^D$.
\end{theorema}

We prove the theorem in Subsection~\ref{ssec:delone}. 

\medskip

Next, the set of lattices $\Lat$ in $G$ is studied. Special attention is drawn to the spaces $\inv(\Lat)\subseteq \inv$ and $\invN(\Lat)\subseteq \invN$ of the corresponding hulls and transversals of lattices $\Lat \cap \Del$ that are $U$-left uniformly discrete and $K$-left relatively dense. The following statement implies that there exist compact spaces of dynamical systems that contain only uniquely ergodic dynamical systems. In combination with the Theorem~\ref{thm:mainapprox}, this is used for regular model sets in Section~\ref{sec:CP-schemes}.

\begin{theorema}
\label{thm:LatCom}
Let $U\subseteq G$ be open and $K\subseteq G$ be compact. Then the topological spaces $\inv(\Lat), \invN(\Lat)$ and $\Lat \cap \Del$ equipped with the corresponding Chabauty-Fell topology are homeomorphic. In particular, $\inv(\Lat)$ and $\invN(\Lat)$ are compact, second countable and Hausdorff spaces and every element in $\inv(\Lat)$ is uniquely ergodic.
\end{theorema}

The proof of this theorem is provided in Subsection~\ref{ssec:spacLat}.

\subsection{Delone sets and topology} \label{ssec:delone}

Let $G$ be a lcsc group then Delone sets are defined as follows.

\begin{definition}
\label{def:LSepRDen}
Let $D \subseteq G$ be a non-empty set. For an open set $U \subseteq G$, and a compact set $K \subseteq G$, we say that 
\begin{itemize}
\item $D$ is {\em $U$-left uniformly discrete} if $\#\big( gU \cap D\big) \leq 1$ for all $g \in G$;
\item $D$ is {\em $K$-left relatively dense} if $D\cdot K = G$. 
\end{itemize}
A $U$-left uniformly discrete and $K$-left relatively dense set $D\subseteq G$ is called {\em $(U,K)$-Delone set} or just {\em Delone set}. The set of all $(U,K)$-Delone sets is denoted by $\Del$.
\end{definition}

Classically, Delone sets are introduced by using a left-invariant metric on the group $G$. According to \cite[Proposition~2.2, Lemma~2.3]{BjHa16}, these notions coincide.

\medskip

Clearly, every element in $\Del$ is a closed subset of $G$ since it is $U$-left uniformly discrete. Thus, $\Del\subseteq\cs(G)$ is a compact, second countable Hausdorff space endowed with the induced Chabauty-Fell topology, cf.\@ \cite{Moo97,Sch00,BaLe04,BjHaPo16}.
Note that due to compactness, we have $\ks(\Del) = \cs(\Del)$. In order to deal with discrete Schr\"odinger operators associated with Delone sets, it is useful to consider $\DelN\subseteq\cs(G)$, the set of Delone sets containing the neutral element $e$. This space is also compact, second countable and Hausdorff in the induced Chabauty-Fell topology.

\medskip

The group $G$ acts naturally by translation homeomorphisms on the space $\Del$, cf.\@ Proposition~\ref{prop:ActClCon}. In order to speak about dynamical systems arising from Delone set, we need the concept of invariant subsets of $\Del$ and $\DelN$. 

\begin{definition}
\label{def:Inv}
Let $U\subseteq G$ be open and $K\subseteq G$ be compact.
\begin{itemize}
\item A subset $\Omega\subseteq\Del$ is called {\em $G$-invariant} 
if $G\cdot D\subseteq \Omega$ for every $D\in\Omega$. The set of all $G$-invariant, 
closed, non-empty subsets of $\Del$ is denoted by $\inv$.
\item A subset $\Tt\subseteq\DelN$ is called {\em invariant} if $\{x^{-1}L\,|\, x\in L\}\subseteq \Tt$ 
for each $L\in \Tt$. The set of all invariant, closed, non-empty subsets of $\DelN$ is denoted by $\invN$.
\end{itemize}
\end{definition}

By using again the Chabauty-Fell topology, the sets $\inv$ and $\invN$ become topological spaces that are the fundamental objects of this work.

\begin{proposition}[\cite{BeckusThesis16,BeBeNi17}]
\label{prop:InvClos}
Let $U\subseteq G$ be open and $K\subseteq G$ be compact. The space $\inv\subseteq\ks(\Del)$ and $\invN\subseteq\ks(\DelN)$ are compact, second countable and Hausdorff in the Chabauty-Fell topology.
\end{proposition}

We will be mainly concerned with $G$-invariant sets arising from one element $D \in\Del$. We define the {\em hull} of $D \in\Del$ as the Chabauty-Fell closure
\begin{eqnarray*}
\Hh^D \;
	:= \; \overline{\big\{ g D \,|\, g \in G\big\}} \;
	\in \; \inv
	\,.
\end{eqnarray*}

Such $G$-invariant closed subsets of $\Del$ are called {\em topologically transitive}. The set
\begin{eqnarray*}
\Tt^D \;
	:= \; \overline{
		\left\{\left.
			x^{-1}D
			\,\right|\, 
			x\in D
		\right\}} \;
	= \; \left\{\left.
			D'\in\Hh^D
			\,\right|\, 
			e\in D'
		\right\} \;
	\in \; \invN	
\end{eqnarray*}
is called {\em transversal of $D$}. In order to prove the Theorem~\ref{thm:TransHullHomeo} we need the following lemma.

\begin{lemma}
\label{lem-Hilfs2}
Let $\Ff\subseteq\Del$ be closed and $\Tt\in\invN$. Then the following assertions are equivalent.
\begin{itemize}
\item[(i)] The intersection $\Ff\cap G\cdot \Tt$ is empty.
\item[(ii)] The intersection $\Ff\cap K^{-1}\cdot \Tt$ is empty. 
\end{itemize} 
\end{lemma}

\begin{proof}
The implication (i)$\Rightarrow$(ii) is clear as $K^{-1}\subseteq G$. Thus, it suffices to show the converse implication. To this end,
let $g\in G$ and $L\in \Tt$. The set $L$ is $K$-left relatively dense (recall this means $L\cdot K=G$), and so $K^{-1}\cdot L^{-1}=(L\cdot K)^{-1} = G^{-1}=G$. Consequently, there is an $x\in L$ and an $h\in K$ such that $h^{-1}x^{-1}=g$. Set $L':=x^{-1}L$. Since $x\in L$, we have $L'\in \Tt$ by the invariance of $\Tt$. Hence, we get $gL=h^{-1}L'\not\in \Ff$, since $\Ff\cap K^{-1}\cdot \Tt=\emptyset$. 
As $g\in G$ and $L\in \Tt$ were arbitrarily chosen, the intersection $\Ff\cap G\cdot \Tt$ is empty. 
\end{proof}

\medskip

{\bf Proof of Theorem~\ref{thm:TransHullHomeo}.} We show  $\Phi$ is (a) well-defined, (b) injective, (c) surjective and (d) continuous. Then \cite[Satz~8.11]{Querenburg2001} implies that $\Phi$ is a homeomorphism, since $\invN$ is compact and $\inv$ is Hausdorff.

(a): Let $\Tt\in\invN$. By definition, $\Phi(\Tt)=G\cdot \Tt\subseteq\Del$ is invariant and non-empty. Thus, it suffices to prove that $G\cdot \Tt$ is Chabauty-Fell closed. Let $g_n L_n\in G\cdot \Tt,\, n\in\NN,$ be convergent to $L\in\Del$. Since $L_n$ is $K$-left relatively dense, we get $K^{-1}\cdot L_n^{-1}=(L_n\cdot K)^{-1}=G^{-1}=G$. Thus, for each $n\in\NN$, there are $x_n\in L_n$ and $h_n\in K$ such that $g_n=h_n^{-1}x_n^{-1}$. By compactness of $K$, there is no loss of generality in assuming that $(h_n)$ converges to $h$ (otherwise pass to a subsequence). Consequently,
we obtain
$$
x_n^{-1}L_n \;
	= \; h_n \big( h_n^{-1}x_n^{-1} L_n\big) \;
	= \; \underbrace{h_n}_{\to h} \big( \underbrace{g_nL_n }_{\to L} \big)
	\underset{n\to\infty}{\longrightarrow} hL
$$
by using the continuity of the $G$-action on $\Del$, c.f. Proposition~\ref{prop:ActClCon}. Since $\Tt$ is closed, the above convergence shows that $h L \in \Tt$. We arrive at $L=h^{-1}(hL) \in G\cdot \Tt$ implying $hL \in G\cdot \Tt$.

(b): Let $\Tt_1, \Tt_2\in\invN$ be such that $\Phi(\Tt_1)=\Phi(\Tt_2)$. It suffices to show $\Tt_1\subseteq \Tt_2$. 
Consider $L\in \Tt_1$. Since $\Tt_1\subseteq \Phi(\Tt_1)=\Phi(\Tt_2)$, there is an element $g\in G$ and $L'\in \Tt_2$ 
such that $gL'=L$. Since $e\in L=gL'$, we have $g^{-1} \in L'$, and from the invariance of $\Tt_2$, we derive $L\in \Tt_2$.

(c): Let $\Omega\in\inv$. Define $\Tt(\Omega):=\{D\in\Omega\,|\, e\in D\}$. Note that $\Tt(\Omega)\subseteq\DelN$ is Chabauty-Fell closed. Additionally, $\Tt(\Omega)$ is invariant in $\DelN$, since $\Omega$ is $G$-invariant. Consequently, $T(\Omega)\in\invN$. We show $\Phi\big(\Tt(\Omega)\big)=\Omega$. Since $\Tt(\Omega)\subseteq\Omega$ and $\Omega$ is $G$-invariant, we have $\Phi\big(\Tt(\Omega)\big)\subseteq\Omega$. For the converse, let $D\in\Omega$. Then for $x\in D$, we have $e\in x^{-1}D$ and $x^{-1}D\in\Omega$, which yields $x^{-1}D\in \Tt(\Omega)$. Thus, $D=xx^{-1}D\in\Phi\big(\Tt(\Omega)\big)$ follows by definition, proving $\Omega\subseteq\Phi\big(\Tt(\Omega)\big)$. 

(d): Let $\us:=\us\big(\Ff,\{\Oo_1,\ldots,\Oo_n\}\big)\subseteq\inv$ be a Chabauty-Fell open set, 
where $\Ff\subseteq \Del$ is compact and the $\Oo_i \subseteq \Del$ are open.
 We intend to show that $\Phi^{-1}(\us)$ is open. It suffices to prove that there exists a Chabauty-Fell open neighborhood $\vs\subseteq\invN$ of $\Tt \in\invN$, which satisfies $\Phi(\Tt)\in\us$, such that $\Phi(\vs)\subseteq\us$. 
Let $\Tt\in\invN$ be so that $\Phi(\Tt)\in\us$. Thus, $\Ff\cap G\cdot \Tt=\emptyset$ and $\Oo_i\cap G\cdot \Tt\neq\emptyset$ for $i=1,\ldots,n$. Note that $\Ff$ is compact
as a subset of $\Del$, and so $K\times \Ff\subseteq G\times\Del$ is compact. Hence, $K\cdot \Ff\subseteq\Del$ is compact since the action of $G$ on $\ks(\Del)=\cs(\Del)$ is jointly continuous. In the same way, one obtains that $G\cdot \Oo_i\subseteq\Del$ is open. Now define the Chabauty-Fell open set $\vs:=\us\big(K\cdot \Ff,\{G\cdot \Oo_1,\ldots,G\cdot \Oo_n\}\big)\cap\invN\subseteq\invN$.
According to Lemma~\ref{lem-Hilfs2}, $\Ff\cap G\cdot \Tt=\emptyset$ is equivalent to $\Ff\cap K^{-1}\cdot \Tt=\emptyset$ being equivalent to $K\cdot \Ff\cap \Tt=\emptyset$. Additionally, $\Oo_i\cap G\cdot \Tt$ is non-empty implying $\Tt\cap G\cdot \Oo_i\neq\emptyset$ for $i=1,\ldots,n$. Consequently, $\vs$ is a Chabauty-Fell neighborhood of $\Tt$. It is left to show that $\Phi(\vs)\subseteq\us$: Let $\Tt'\in\vs$. Then Lemma~\ref{lem-Hilfs2} implies $\Ff\cap G\cdot \Tt'=\emptyset$. Furthermore, $\Tt'\cap G\cdot \Oo_i\neq\emptyset$ leads to $\Oo_i\cap G\cdot \Tt'\neq\emptyset$. Altogether, $\Phi(\Tt')\in\us$ follows finishing the proof. 
\hfill$\Box$ 

\medskip

A crucial ingredient for the convergence of the spectral quantities like the density of states and the autocorrelation is that the corresponding integrands are continuous. The following assertion provides this continuity.

\begin{lemma}
\label{lem:SumDCont}
Let $G$ be a lcsc group and $\Del$ be the compact space of Delone sets in $G$ with $e\in U\subseteq G$ open and $K\subseteq G$ compact. For every $\varphi\in\Cc_c(\Del\times G)$, the map
$$
\Phi(\varphi):\Del\to \CC\,,\;
	\Phi(\varphi)(D) \; := \; \sum_{x\in D} \varphi(D,x)
	\,,
$$
is continuous in the Chabauty-Fell topology of $\Del$.
\end{lemma}

\begin{proof}
Let $D\in\Del$ and $\varepsilon>0$. Denote by $F$ the compact set $\pi_G(\supp(\varphi))\subseteq G$ where $\pi_G:\Del\times G\to G$ denotes the projection on $G$. Since $D$ is $U$-uniformly discrete, we have $D\cap F = \{x_1,\ldots, x_l\}$ and so $\Phi(\varphi)(D)=\sum_{i=1}^l \varphi(D,x_i)$. Due to continuity of $\varphi$, there exists an open neighborhood $\vs_\varepsilon$ of $D$ and open neighborhoods $U_\varepsilon^i\subseteq x_iU$ of $x_i$ for $i=1,\ldots, l$ such that
$$
\big| \varphi(D',y)-\varphi(D,x_i) \big|
	< \frac{\varepsilon}{l}
	\,,\qquad
	(D',y)\in \vs_\varepsilon\times U_\varepsilon^i
	\,.
$$
Define the Chabauty-Fell open neighborhood $\us:=\vs_\varepsilon\cap\us\big(C,\os)$ for the compact set $C:=F\setminus\bigcup_{i=1}^l U_\varepsilon^i\subseteq G$ and the finite family of open sets $\os:=\{U_\varepsilon^1,\ldots,U_\varepsilon^n\}$. 

Let $D'\in\us$. Since $U_\varepsilon^i\subseteq x_iU$ and $D'$ is $U$-uniformly discrete, the intersection $D'\cap U_\varepsilon^i=\{y_i\}$ contains exactly one element. Thus, we have $D'\cap F\subseteq\{y_1,\ldots,y_l\}$ by construction of $\us$. Since $\varphi(D',y)=0$ if $y\in D'\cap F^c$, the identity $\Phi(\varphi)(D')=\sum_{i=1}^l \varphi(D',y_i)$ follows. Hence,
$$
\left|
	\Phi(\varphi)(D)-\Phi(\varphi)(D')
\right|
	\leq \sum_{i=1}^l \underbrace{|\varphi(D,x_i)-\varphi(D',y_i)|}_{\leq\varepsilon/l}
	\leq \varepsilon
	\,,\qquad
	D'\in\us\,,
$$
as $(D',y_i)\in\vs_\varepsilon\times U_\varepsilon^i$ for $i=1,\ldots,l$. Since $\varepsilon>0$ was arbitrary, this shows the desired continuity of $\Phi(\varphi)$.
\end{proof}


\subsection{Spaces of lattices} \label{ssec:spacLat}
In this section, we provide some elementary and for the most
part well-known facts about dynamics on lattices. 
Note that those constitute a specific class of Delone dynamical systems. 
Precisely, we study 
continuous actions of a lcsc unimodular group $G$ on homogeneous spaces of the form $G/\Gamma$, where $\Gamma$ is a cocompact lattice in $G$. 

\medskip

Recall that a closed, discrete subgroup $\Gamma< G$ is called {\em lattice} if the quotient $G/\Gamma$ admits a $G$-invariant probability measure $m_Y$. 
We say a lattice is {\em uniform} or {\em cocompact} if $G/\Gamma$ is compact. 
It is not hard to see that in the cocompact situation, there can be at most one invariant probability measure $m_Y$ on $G/\Gamma$. Accordingly, we refer to $m_Y$ as the {\em Haar measure} on the quotient space, even if the latter does not need to be a group. 

\medskip

The set of all lattices in $G$ is denoted by $\Lat$. 
Recall that if a lcsc group admits a lattice, then it must be {\em unimodular}, i.e.\@ the left Haar measure and the right Haar measure coincide on $G$, cf.\@ e.g.\@ \cite[Theorem~9.1.6]{DE09}. The space of cocompact closed subgroups of $G$ and the space of discrete subgroups are open in the Chabauty-Fell topology \cite{Bou63,Ole73,BrHaKl09}. 
The space of lattices does not need to be open in general. However, there exist examples with open lattice spaces $\Lat$ \cite{Rag72}. 
The key to this issue are uniform bounds on the discreteness as well as on the compactness parameters of the lattices. 
With bounds on those parameters, one obtains subspaces which are compact in the Chabauty-Fell topology. For Euclidean lattices, 
this phenomena is known as the {\em Mahler criterion}, cf.\@ \cite{Mah46},
which was generalized to semisimple Lie groups by Mumford \cite{Mum71}.
In our situation, the space $\Lat\cap \Del$ equipped with the induced Chabauty-Fell topology is a compact, second countable Hausdorff space, c.f. \cite{Cha50}.

\begin{lemma}
\label{lem:LatDel}
Every cocompact lattice $\Gamma< G$ is a Delone set. Furthermore, the transversal satisfies $\Tt^\Gamma=\{\Gamma\}$ and the hull $\Hh^\Gamma$ equals the orbit $G\cdot\{\Gamma\}$.
\end{lemma}

\begin{proof}
Since $\Gamma^{-1}\Gamma=\Gamma$ and $\Gamma$ is discrete, \cite[Proposition~2.2]{BjHa16} implies that $\Gamma$ is uniformly discrete. Additionally, $\Gamma^{-1}\Gamma=\Gamma$ yields $\Tt^\Gamma=\{\Gamma\}$. By cocompactness, the hull $\Hh^\Gamma$ is equal to $G\cdot\{\Gamma\}$. 
Hence, $\emptyset\not\in\Hh^\Gamma$, which in turn implies that $\Gamma$ is relatively dense \cite[Proposition~4.4]{BjHa16}.
\end{proof}

\medskip

The following proposition is probably known. 
Similar continuity results play a significant role in the theory of
invariant random subgroups (IRS). For a nice survey on IRS, we refer to \cite{Gel15}. 
In the context of this work, we give a proof via the abstract approximation Theorem~\ref{thm:mainapprox}. 
This underscores our point of view of interpreting convergence of lattices
and equivalently convergence of their hulls as a particular 
case of convergence of Delone dynamical systems. 
However, this result becomes important when 
we study continuity behavior of spectral quantities over model sets with respect to
small changes of the projection lattice, cf.\@ Section~\ref{sec:CP-schemes}.

\begin{proposition}
\label{prop:weakastlattice}
Let $\Gamma_n,\Gamma\in\Lat	 \cap \Del$ for $n\in\NN$. Then $\Gamma_n\to\Gamma$ in $\Lat\cap \Del$ if and only if $\{\Gamma_n\}\to\{\Gamma\}$ in $\invN$. In particular, we have in this situation
$$
\lim_{n \to \infty} m_n = m_Y
$$
in the weak-$\ast$-topology on $\Del$, where $m_n$, respectively $m_Y$ denote the
Haar measures on $G/\Gamma_n$, respectively on $G/\Gamma$. 
\end{proposition}

\begin{proof}
Suppose $\{\Gamma_n\}\to\{\Gamma\}$ in $\invN$ and let $\us$ be an open neighborhood of $\Gamma$ in $\Del$. Then $\vs(\emptyset,\{\us\})$ is an open neighborhood of $\{\Gamma\}$. Hence, there is an $n_0\in\NN$ such that $\{\Gamma_n\}\in\vs(\emptyset,\{\us\})$ follows for $n\geq n_0$. Thus, $\Gamma_n\in\us$ is derived for $n\geq n_0$ implying $\Gamma_n\to\Gamma$.\\
Now suppose $\Gamma_n\to\Gamma$ in $\Del$ and let $\vs(F,\os)$ be an open neighborhood of $\{\Gamma\}$. Since $\{\Gamma\}\in\vs(F,\os)$, the set $\us:=\bigcap_{O\in\os} O \cap \Del\setminus F$ is an open neighborhood of $\Gamma$ in $\Del$. Hence, there is an $n_0\in\NN$ such that $\Gamma_n\in\us$ for $n\geq n_0$ implying $\{\Gamma_n\}\in\vs(F,\os)$. Thus, the convergence $\{\Gamma_n\}\to\{\Gamma\}$ in $\invN$ follows.\\
Since $m_Y$ is the unique invariant probability measure on $G/\Gamma$, the claimed convergence of the measures follows from Theorem~\ref{thm:mainapprox}. 
\end{proof}

\medskip

Since $\Del$ is closed in the Chabauty-Fell topology we further derive the following. Denote by $\invN(\Lat)\subseteq\invN$ the set of all transversals associated with the elements of $\Lat \cap \Del$.

\medskip

{\bf Proof of Theorem~\ref{thm:LatCom}.}
Theorem~\ref{thm:TransHullHomeo} implies that $\inv(\Lat)$ and $\invN(\Lat)$ are homeomorphic if equipped with the Chabauty-Fell topology. According to Lemma~\ref{lem:LatDel}, every element in $\invN(\Lat)$ contains exactly one element. Define the map $\Psi:\Lat\cap\Del\to \invN(\Lat)$ by $\Gamma\mapsto\{\Gamma\}$ being bijective and continuous by Proposition~\ref{prop:weakastlattice}. Furthermore, $\invN(\Lat)\subseteq\invN$ is Hausdorff and $\Lat\cap\Del\subseteq\Del$ is compact in the corresponding induced Chabauty-Fell topologies. Thus, $\Psi$ is a homeomorphism, see e.g. \cite[Satz~8.11]{Querenburg2001}. Thus, $\inv(\Lat)$ and $\invN(\Lat)$ are compact, second countable and Hausdorff spaces. \hfill$\Box$


\section{Convergence of the autocorrelation} \label{sec:ConAC}

We introduce weighted Delone sets in Subsection~\ref{ssec:AC}. Their definition is motivated by weighted model sets that were studied in the abelian world, \cite{Baa01,BaMo04,LeRi07,BaakeGrimm13,Str14,RiSt15,RiSt17,BaHuSt17}. However, our point of view is slightly different since we will consider weighted model sets as subsets of the product space $G \times H$ without projecting them to $G$, where both $G,H$ are unimodular lcsc groups. The weights are given by compactly supported functions on $H$. In that respect, our approach is more closely related to the work \cite{KeRi16}.

\medskip

In Subsection~\ref{sec:Appr-AC}, we show the main assertion of this section that the autocorrelation converges in the weak-$\ast$ topology if the Delone dynamical systems converge in the Chabauty-Fell topology and the limiting dynamical system is uniquely ergodic, c.f. Theorem~\ref{thm:Appr-wgtAC} and Corollary~\ref{cor:Appr-AC}. The proof is an application of Theorem~\ref{thm:mainapprox}.



\subsection{The weighted autocorrelation for Delone sets} \label{ssec:AC}

Throughout the section $\Delr$ denotes the set of Delone sets with fixed parameters $U$ and $K$. The closed, $G\times H$-invariant subsets are denoted by $\IDelGH$. Whenever we only speak about Delone sets in $G$ we will specify. We denote by $\pi_G, \pi_H$
the continuous projections from $G \times H$ to $G$, respectively to $H$.

\begin{definition}[Weighted Delone set]
\label{def:WghtModSet}
A measurable, bounded function $\win:H\to[0,\infty)$ with compact support is called {\em window function}. If, in addition, $\win$ is continuous, we say $\win$ is a continuous window function. For a Delone set $D$ in $G \times H$, we call the tuple $(D,\win)$  {\em weighted Delone set}. If $D=\Gamma$ is a cocompact lattice, we call $(\Gamma,\win)$ a {\em weighted model set}.
\end{definition}

If $(D, \win)$ is a weighted Delone set, then every point $x \in D$ carries a weight given by the value $\win\big( \pi_H(x) \big)$. 
We point out that for the definition of weighted model sets, there is no need to restrict oneself to cocompact lattices. 
However, in order to stay in the setting of weighted Delone sets, we include cocompactness in the definition.

\begin{definition}[Periodization]
\label{def:Period}
For a window function $\win:H\to[0,\infty)$, we define the {\em (weighted) periodization} on the Delone sets $\Delr$ on $G\times H$ with fixed parameters by
$$
\Pp_{\win}f: \Delr \to \RR \,,\quad
	\Pp_{\win}f(D)
		:= \sum_{x \in D}
			f(\pi_G(x))\,\win(\pi_H(x)), 
		\quad\quad f\in \Cc_c(G) .
$$
\end{definition}

\begin{lemma}
\label{lem:PerProp}
The periodization is $G$-equivariant and the map $\IDelGH\ni D\mapsto \Pp_{\win}f(D)$ is continuous for each $f\in\Cc_c(G)$ if $\win$ is a continuous window function.
\end{lemma}

\begin{proof}
Since $\win(\pi_H(x))=\win(\pi_H(gx))$ for all $g\in G$ and $x\in G\times H$, the $G$-equivariance $\Pp_{\win}f(g\Lambda)=\Pp_{\win}f(g \cdot)(\Lambda)$ follows by a short computation. Let $\win$ be a continuous window function and $f\in\Cc_c(G)$. Then the map $\Delr\times G\times H\ni (D,x)\mapsto f(\pi_G(x))\,\win(\pi_H(x))$ is continuous with compact support contained in $\Delr\times\supp(f)\times\supp(\win)$. Thus, the continuity of the periodization follows by Lemma~\ref{lem:SumDCont}.
\end{proof}

\medskip

The notion of periodization (sometimes also called Siegel transformation) has been studied before in a wide range of contexts, among them diffraction theory. 
Here, periodization maps can be used to determine the autocorrelation of Delone dynamical system, which is a unique positive definite Radon
measure on $G$. 
Due to its continuity and $G$-invariance, the window function is an additional compatible ingredient. 

\medskip

For the proof of existence, we use that $\Cc_c(G)$ is a $\ast$-algebra with the usual involution and convolution defined by
$$
f^*(g) 
	:= \overline{f(g^{-1})}, \qquad
(f_1\ast f_2)(g) 
	:= \int_G f_1(h) f_2(h^{-1}g) \, dm_L(h).
$$
The following proposition is a well-known statement from abstract harmonic analysis.
Our proof follows the lines of the proof of \cite[Proposition~4.8]{BjHaPo16}. 

\begin{proposition} 
\label{prop:ACweight}
Fix $\Omega\in\IDelGH$ with a $G$-invariant, probability measure $\mu$ on $\Omega$ and suppose that $\win\in\Cc_c(H)$ is a continuous window function. Then, there is a unique positive definite Radon measure $\gamma:=\gamma_{\Omega,\win,\mu}$ on $G$ satisfying 
$$
\gamma\big( f_1^{*} * f_2 \big) 
	= \big\langle \Pp_{\win}f_1, \Pp_{\win}f_2 \big\rangle_{L^2(\Omega, \mu)}
	= \int_\Omega \overline{\Pp_{\win}f_1(D)}\; \Pp_{\win}f_2(D) \; d\mu(D)
$$
for all $f_1,f_2 \in \Cc_c(G)$.
\end{proposition} 

\begin{proof}
Define a Radon measure $\tilde{\gamma}$ on $G\times G$ by
$$
\tilde{\gamma}(\overline{f_1}\otimes f_2)
 	:=  \big\langle \Pp_{\win}f_1, \Pp_{\win}f_2 \big\rangle_{L^2(\Omega, \mu)},
 	\qquad
 	f_1,f_2\in\Cc_c(G).
$$
The integral is well-defined by Lemma~\ref{lem:PerProp}. Additionally, Lemma~\ref{lem:PerProp} and the $G$-invariance of $\mu$ imply that the measure $\tilde{\gamma}$ is invariant by the action of the diagonal subgroup $\Delta(G)$ of $G\times G$. Thus, we can view $\tilde{\gamma}$ as a measure on $\Delta(G)\backslash G\times G$. The map $\Delta(G)\backslash G\times G\to G, [g,h]\mapsto g^{-1}h$ is a homeomorphism. Consequently, $\tilde{\gamma}$ corresponds to a measure $\gamma$ satisfying $\tilde{\gamma}(\overline{f_1}\otimes f_2) = \gamma(f_1^{*} * f_2)$. The identity $\gamma\big( f^{*} * f \big) =\|\Pp_{\win}f\|_{L^2(\Omega, \mu)}^2\geq 0$ implies that the measure is  positive definite. Furthermore, $\{f_1^{*} * f_2\,|\, f_1,f_2\in\Cc_c(G)\}\subseteq\Cc_c(G)$ is dense implying the uniqueness.
\end{proof}

\medskip


\begin{definition}
\label{def:ACweight}
Let $\Omega\in\IDelGH$ with a $G$-invariant, probability measure $\mu$ on $\Omega$ and $\win\in\Cc_c(H)$ be a continuous window function. We call the measure $\gamma_{\Omega,\win,\mu}$ given by the formula in Proposition~\ref{prop:ACweight} a 
{\em ($\win$-weighted) autocorrelation of $(\Omega,\mu,G)$}.\\
If, in addition, the dynamical system $(\Omega,G)$ is uniquely 
ergodic with invariant measure $\mu$, then the measure $\gamma_{\Omega,\win}:=\gamma_{\Omega,\win,\mu}$ is called {\em the} {\em ($\win$-weighted) autocorrelation measure} or just the {\em ($\win$-weighted) autocorrelation} of $\Omega$. If $H=\{e_H\}$ is the trivial group and $\win= 1$ the notation $\gamma_{\Omega,\mu}$ respectively $\gamma_\Omega$ is used. 
\end{definition}

\begin{remark}
\label{rem:ACweight}
It is worth pointing out that the previously defined concept delivers the notion of autocorrelation of Delone sets $D$ in a lcsc group $G$ by considering the trivial group $H=\{e_H\}$ with continuous window function $\win= 1$. It was shown in \cite[Corollary~5.4]{BjHaPo16} that for uniquely ergodic hulls $\Hh^D$, this notion coincides with the classical concepts of autocorrelation. 
\end{remark}



\subsection{Approximation of the autocorrelation} \label{sec:Appr-AC}

We now prove that the autocorrelation of uniquely ergodic Delone dynamical systems
is approximated by all possible autocorrelations of hulls converging to the system.

\begin{theorem}
\label{thm:Appr-wgtAC}
Let $\Omega\in\IDelGH$ be uniquely ergodic with measure $\mu$ and $\win\in\Cc_c(H)$ be a continuous window function. If $(\Omega_l,\mu_l,G)$
are Delone dynamical systems such that $\Omega_l\in\IDelGH\,,\, l\in\NN\,,$ and $\Omega_l \to \Omega$ 
in the Chabauty-Fell topology, then
$$
\lim_{l \to \infty} \gamma_{\Omega_l,\win,\mu_l} = \gamma_{\Omega,\win}
$$
in the weak-$*$-topology on $\mathcal{M}(G)$.
\end{theorem} 

\begin{proof}
Note first that $\supp\,\mu \subseteq \Omega$ and $\supp\,\mu_l \subseteq \Omega_l$. 
Since $\Omega$ is uniquely ergodic and $\Omega_l\to\Omega$ in the Chabauty-Fell topology, Theorem~\ref{thm:mainapprox} implies the convergence of the measures $(\mu_l)$ in the weak-$\ast$ topology to $\mu$. We write $\gamma_l := \gamma_{\Omega_l,\win,\mu_l}$ and $\gamma:=\gamma_{\Omega,\win}$. The continuity of $\win$ leads to the continuity of $D\mapsto\Pp_\win (f)(D)$ for $f\in\Cc_c(G)$, c.f. Lemma~\ref{lem:PerProp}. Thus, the formula of Proposition~\ref{prop:ACweight} and $\mbox{w-}*\mbox{-}\lim_{l\to\infty}\mu_l=\mu$ leads to
$$
\lim_{l \to \infty} \gamma_l\big( f_1^{*} * f_2 \big) 
	= \lim_{l \to \infty} \int_{\Delr} \overline{\Pp_{\win}(f_1)}\; {\Pp_{\win}(f_2)}  \,d\mu_l 
	= \int_{\Delr} \overline{\Pp_{\win}(f_1)}\; {\Pp_{\win}(f_2)}  \,d\mu 
	= \gamma \big( f_1^{*} * f_2 \big)
$$
for $f_1,f_2\in\Cc_c(G)$. 
\end{proof}

\medskip

The specific case where $H=\{e_H\}$ is the trivial group and $\win= 1$, leads to the following. 

\begin{cor}[Continuity of the autocorrelation for Delone sets]
\label{cor:Appr-AC}
Let $\Omega\in\IDelG$ be a uniquely ergodic Delone dynamical system. 
If $(\Omega_l,\mu_l,G)$
are Delone dynamical systems in $\IDelG$ such that $\Omega_l \to \Omega$ 
in the Chabauty-Fell topology, then
$$
\mbox{w-}*\mbox{-}\lim_{l \to \infty} \gamma_{\Omega_l,\mu_l} = \gamma_{\Omega,\mu}
$$
in the weak-$*$-topology on $\mathcal{M}(G)$.
\end{cor}


\section{Random bounded operators and convergence of the density of states} \label{sec:RanOpDOS}

This section deals with random bounded operators that are introduced in Subsection~\ref{ssec:RandOp}. As a further application of Theorem~\ref{thm:mainapprox}, we derive that the corresponding density of states converges in the weak-$\ast$ topology if the underlying Delone dynamical systems converge and the limiting Delone dynamical system is uniquely ergodic, c.f. Theorem~\ref{thm:Appr-wgtDOS} and Corollary~\ref{cor:Appr-DOS}.


\subsection{Random bounded operators} \label{ssec:RandOp}

As before, $\Delr$ denotes the set of Delone sets in $G\times H$ with fixed parameters $U$ and $K$. The closed, $G\times H$-invariant subsets are denoted by $\IDelGH$. Whenever we only speak about Delone sets in $G$ we will specify. In order to use the abstract theory of the density of states discussed in the Appendix~\ref{ssec:appendix-Groupoid}, we assume that the groups $G$ and $H$ are unimodular.

\medskip

In the following we describe the class of operators which lie at the heart of our investigations. Let $\win\in\Cc_c(H)$ be a window function and $D\in\Delr$. A priori, $\langle u, v\rangle_{D,\win} := \sum_{x\in D} \overline{u(x)}\, v(x)\, \win(\pi_H(x))$ defines only a semi inner product. Passing to $\tilde{D}:=\{ x\in D \,|\, \win(\pi_H(x))\neq 0 \}$, the Hilbert space $\ell^2(\tilde{D},\win)$ can be defined and the above semi inner product is actually an inner product. With a slight abuse of notation, we consider $\ell^2(D,\win)$ as Hilbert space with inner product $\langle u, v\rangle_{D,\win} := \sum_{x\in D} \overline{u(x)}\, v(x)\, \win(\pi_H(x))$. Hence, the window function $\win\in\Cc_c(H)$ induces a family of weighted Hilbert spaces $\big(\ell^2(D,\win)\big)_{D\in\Delr}$. The $G$-invariance of the weight function $\win$ will play a crucial role in order to define the direct integral space $\int_\Omega^\oplus \ell^2(D,\win)\, d\mu(D)$.

\medskip

A crucial point of view is that the Hamiltonians are elements of the associated $C^\ast$-algebra \cite{Bel86}. Following the lines of \cite[Section~4]{LeSt03-Algebras}, we define a $C^\ast$-algebra of operators that we are dealing with. These operator families are integral operators with suitable kernels (or convolution operators) defined on the family of Hilbert space $\big(\ell^2(D,\win_D)\big)_{D\in\Delr}$. 

\begin{lemma}
\label{lem:kerXclos}
Let $W\subseteq H$ be compact. The set
$$
X \; := \; X_W \; := \;
	\big\{
		(x,D,y)
		\;\big|\; 
		D\in \Delr,\;
		x,y\in D\cap G\times W
	\big\}
$$
is a closed $G$-invariant subset of $(G\times H)\times\Delr\times (G\times H)$.
\end{lemma}

\begin{proof}
The set $X$ is closed as intersection of the two closed sets $\{(x,D,y) \,|\, x,y\in D\}$ and $(G\times W) \times \Delr\times (G\times W)$. 
\end{proof}

\medskip

In the following, the closed subset $X\subseteq (G\times H)\times\Delr\times (G\times H)$ is equipped with the induced topology.

\begin{definition} 
\label{def:kernel}
A continuous function $a\in\Cc(X)$ is called
{\em kernel of finite type} if
\begin{itemize}
\item[(i)] $\|a\|_\infty<\infty$; 
	\hfill \textbf{(bounded)}\\[-0.2cm]
\item[(ii)] $\begin{array}{l}
		\text{there is a compact } K_a\subseteq G \text{ such that } K_a^{-1}=K_a \text{ and for all}\\[0.05cm]
		(x,D,y)\in X \text{ with } \pi_G(x^{-1}y)\not\in K_a \,, \text{ we have } a(x,D,y)=0;
	\end{array}$ 
	\hfill \textbf{(finite range)}\\[-0.01cm]
\item[(iii)] $a(gx,gD,gy)=a(x,D,y)$ for all $g\in G$ and $(x,D,y)\in X$.
	\hfill \textbf{($G$-invariance)}
\end{itemize}
The set $K_a$ is called {\em support of influence of $a$}.
\end{definition}

In order to define an associated $C^\ast$-algebra, let $\win\in\Cc_c(H)$ be a continuous window function and $X:=X_{\supp(\win)}$. The vector space $\Kfin$ of all kernels $a\in\Cc(X)$ of finite type (with pointwise addition and scalar multiplication) give rise to a $\ast$-algebra if equipped with
\begin{eqnarray*}
a^*(x,D,y) := &\overline{a(y,D,x)}\,,
	\qquad &\textbf{(involution)}\\
(a\star b)(x,D,y) := &\sum\limits_{z\in D} 
		a\big(x,D,z\big)\; b\big(z,D,y\big)\; \win(\pi_H(z))\,,
	\qquad &\textbf{(convolution)}
\end{eqnarray*}
where $\overline{a}$ denotes complex conjugation. It is straight forward to check that $a^*$ and $a\star b$ are again elements of $\Kfin$. For the continuity of $a\star b$ one uses the finite range property and an argument similar to the one in the proof of Lemma~\ref{lem:SumDCont}.

\medskip

We define a $\ast$-representation $\lambda^D:\Kfin\to\Ll(\ell^2(D,\win)),\, D\in\Delr,$ of $\Kfin$ into the linear bounded operators of $(\ell^2(D,\win))_{D}$ via
$$
\big(\lambda^D(a) u\big)(x) \; := \;
	\sum_{y\in D} a(x,D,y)\; u(y)\; \win(\pi_H(y))\,,
$$
for $a\in\Kfin$, $u\in\ell^2(D,\win)$ and $x\in D$. Recall that $\lambda^D$ is a $\ast$-representation if it is linear, multiplicative (i.e., $\lambda^D(a\star b)=\lambda^D(a)\lambda^D(b)$) and compatible with the involution (i.e., $\lambda^D(a^*)=\lambda^D(a)^*$). Furthermore, the family of representations $\big(\lambda^D\big)_D$ is faithful, i.e., $a=0$ if and only if $\lambda^D(a)=0$ for all $D\in\Delr$. The norm completion $\Cfin$ of $\Kfin$ with the norm $\|a\|:= \sup_{D\in\Delr} \|\lambda^D(a)\|$ gets a $C^\ast$-algebra.

\medskip

\begin{remark}
\label{rem:groupoidCalg}
As shown for $G=\RR^d$ in \cite[Proposition~4.4]{LeSt03-Algebras}, the $C^\ast$-algebra $\Cfin$ is isomorphic in terms of $C^\ast$-algebra of the associated groupoid $C^\ast$-algebra studied in \cite{Bel86,BeHeZa00,BeBeNi17}. There, the groupoid $C^\ast$-algebra is defined by the transversal of a Delone dynamical system. Those results extend to our situation. We leave the proof to the reader as it is straightforward.
\end{remark}

\begin{definition} 
\label{def:RanOp}
The operator family $A:=(A_D)_{D}$ defined by $A_D:= \lambda^D(a)$ for $a\in\Cfin$ is called a {\em random bounded operator} over the family of Hilbert spaces $(\ell^2(D,\win))_{D}$. We say that $a\in\Cfin$ is the {\em kernel} of $A$. If $a\in\Kfin$, we say $A$ is of finite range.
\end{definition}

Let $u\in\Cc_c(G\times H)$ and define $(u_D)_{D}$ by the restrictions $u_D:=u|_D$ for $D\in\Delr$. The following statement assures that random bounded operators are contained in the von Neumann algebra defined in Appendix~\ref{ssec:appendix-Groupoid}. With this at hand, the corresponding theory about the density can be used for random bounded operators.

\begin{proposition}
\label{prop:RanOpProp}
Let $A$ be a random bounded operator with kernel $a\in\Cfin$.
\begin{itemize}
\item[(i)] Let $\rho\in\Cc_c(G)$. Then the map
	$
		\Delr\ni D\mapsto \sum_{x\in D} \rho\big(\pi_G(x)\big)\langle \delta_x , A_D \delta_x \rangle_{D,\win}
	$
	is continuous.
\item[(ii)] The operator family $A$ is $G$-equivariant, i.e., for every $g\in G$ and $D\in\Delr$, we have
	$
		T_g A_D T_g^* = A_{gD}
	$
	where 
	$T_g:\ell^2(D,\win)\to\ell^2(gD,\win),$ $T_g u := u(g^{-1}\cdot)$.
\end{itemize}
\end{proposition}

\begin{proof}
We show the desired results for random bounded operators $A$ of finite range. They extend immediately to all random bounded operators by approximating them in norm. So let $a\in\Kfin$ be a kernel of $A$ with support of influence $K_a\subseteq G$. Using Lemma~\ref{lem:kerXclos}, $X\subseteq (G\times H)\times\Delr\times (G\times H)$ equipped with the induced topology is second countable, locally compact, Hausdorff. Thus, $X$ is a normal space and so Tietzes extension theorem applies. Specifically, there is a $\tilde{a}\in\Cc\big(G\times H\times\Delr\times G\times H\big)$ satisfying $\tilde{a}|_X=a$.

(i): The map $\varphi_A:\Delr\times G \times H\to\CC$ defined by 
$$
\varphi_A(D,x) 
	:= \rho\big(\pi_G(x)\big)\, 
		\tilde{a}(x,D,x)\, 
		\win(\pi_H(x))
	,\qquad (D,x)\in \Delr\times G \times H
$$
is continuous as all involved functions are continuous. Furthermore, its support $\supp(\varphi_A)$ is contained in the compact set $\Delr\times\supp(\rho)\times\supp(\win)$. Hence, the map 
$$
\Delr\ni D 
	\mapsto  \sum_{x\in D} \varphi_A(D,x) 
	= \sum_{x\in D} \rho\big(\pi_G(x)\big)\, \tilde{a}(x,D,x)\, \win(\pi_H(x))
$$ 
is continuous by Lemma~\ref{lem:SumDCont}. The latter sum equals to $\sum_{x\in D} \rho\big(\pi_G(x)\big)\langle \delta_x , A_D \delta_x \rangle_{D,\win}$ since $\tilde{a}$ agrees with $a$ on $X$. Consequently, we have shown the desired continuity.

(ii): Let $g\in G$ and $D\in\Delr$. For $x\in gD$ and $u\in\Cc_c(G\times H)$, we get
\begin{align*}
\big(T_g A_D T_g^*u\big)(x)
	\; =\; &\sum_{y\in D} 
		a\big( g^{-1}x,D,y \big)\,
		u(gy)\, 
		\win(\pi_H(y))\\
	\; =\; &\sum_{y\in D} 
		a(x,gD,gy)\, 
		u(gy)\,
		\win(\pi_H(gy))
	\; =\; \big(A_{gD}u\big)(x)
\end{align*}
by using the $G$-invariance of the kernel $a$ and the window function $\win\circ\pi_H$. Since $\{u_D \,|\, u\in\Cc_c(G\times H)\}=\Cc_c(D,\win) \subseteq \ell^2(D,\win)$ is dense and $A$ is bounded, the desired equivariance of the operators follows.
\end{proof}

\begin{remark}
The specific case that $H=\{e_H\}$ is the trivial group with $\win= 1$ provides the class of random bounded operators $A=(A_D)_{D\in\DelrG}$.
\end{remark}

Based on the non-commutative integration theory of A. Connes \cite{Con78}, the density of states is introduced in \cite{LeSt03-Algebras,LeSt03-Delone,LePeVe07} in the setting of von Neumann algebras arising from a groupoid structure. Dealing with continuous function, it can likewise be defined in the realm of $C^{\ast}$-algebras. In order to obtain
a measure preserving groupoid structure (which is part of the requirements of an {\em admissible setting}, cf.\@ \cite{LePeVe07}), we need to assume that the
group $G$ is unimodular. A more detailed discussion is provided in Appendix~\ref{ssec:appendix-Groupoid}.

\medskip

Let $\Omega\in\IDelGH$ with $G$-invariant probability measure $\mu$ on $\Omega$. We define the associated von-Neumann algebra $\Nn(\Omega,G\times H,\mu,\win)$ of all measurable, bounded and $G$-equivariant operator families $(A_D)_{D\in\Omega}$ (quotient by the equivalence relation that the operator families agree $\mu$-almost everywhere). For this construction, it is crucial that the window function is $G$-invariant in order to use the theory of direct integrals for $\int_\Omega^\oplus \ell^2(D,\win)\, d\mu(D)$. For more information on direct integral theory, we refer the reader to \cite{Dixmier81}.

\begin{corollary}
\label{cor:RaOpvNeu}
Every random bounded operator $A$ defines an element of $\Nn(\Omega,G\times H,\mu,\win)$ and so $A_\Omega=\int^\oplus_\Omega A_D \, d\mu(D)$ is a diagonalizable operator on $\int_\Omega^\oplus \ell^2(D,\win)\, d\mu(D)$.
\end{corollary}

\begin{proof}
Consider the restriction $A_\Omega:=(A_D)_{D\in\Omega}$ of $A$ to $\Omega$. Then $\sup_{D\in\Omega}\|A_D\|\leq \|A\|$ holds and so $A_\Omega$ is bounded. Furthermore, $A_\Omega$ is $G$-equivariant by Proposition~\ref{prop:RanOpProp}~(ii). Finally, $\Omega\ni D\mapsto\langle A_D u_D , u_D \rangle_{D,\win}$ is continuous (and hence measurable) for all $u\in\Cc_c(G\times H)$ which follows the same lines as Proposition~\ref{prop:RanOpProp}~(i). Thus,  $A_\Omega$ is measurable on the induced measurable structure on $\int^\oplus_\Omega \ell^2(D,\win)\, d\mu(D)$ defined in Appendix~\ref{ssec:appendix-Groupoid}. Consequently, the equivalence class of $A_\Omega$ is in $\Nn(\Omega,G\times H,\win)$ and so it is diagonalizable. 
\end{proof}

\begin{remark}
\label{rem:RaOpvNeu}
Corollary~\ref{cor:RaOpvNeu} asserts that $\Cfin$ is a $C^\ast$-subalgebra of the von Neumann algebra $\Nn(\Omega,G\times H,\mu,\win)$ defined in the appendix, c.f. Appendix~\ref{ssec:appendix-Groupoid}.
\end{remark}

For self-adjoint, diagonalizable operators, we now define their weighted density of states in terms of a Pastur-Shubin trace formula.

\begin{definition}
\label{def:wgtDOS}
Let $G$ and $H$ lcsc groups where $G$ is unimodular and $\Omega\in\IDelGH$ with $G$-invariant probability measure $\mu$. Consider a continuous window function $\win\in\Cc_c(H)$ and $\rho \in \Cc_c(G)$ a non-negative function satisfying $\int_G \rho\, dm_G = 1$. The {\em (weighted) abstract density of states} of a self-adjoint random bounded operator $A$ is defined by
$$
\eta^A_{\Omega,\win,\mu} \big( \phi \big) 
	= \int_{\Omega} 
		\sum_{x \in D} \rho(x)\; \langle \delta_x, \phi(A_D)\delta_x \rangle_{D,\win}
	\,d\mu(D),
	\qquad
	\phi\in\Cc_c(\RR).
$$
\end{definition}

The notion above is well-defined in the sense that it is independent of the choice of $\rho$ by Proposition~\ref{prop:GenTrace}. The proof follows by standard computations \cite[Lemma~2.9]{LePeVe07} in context von Neumann algebras. For the convenience of the reader, we provide a proof of this fact in Appendix~\ref{ssec:appendix-Groupoid}. If $\Omega\in\IDelGH$ is uniquely ergodic with measure $\mu$, we use the notation $\eta^A_{\Omega,\win}:=\eta^A_{\Omega,\win,\mu}$. Furthermore, if $H=\{e_H\}$ is the trivial group and $\win= 1$, the index $\win$ is not used in the sequel.


\subsection{Approximation of the density of states for Delone sets} \label{ssec:DOS}

Next, we show that the abstract density of states can be approximated in the weak-$\ast$ topology by approximating the uniquely ergodic Delone dynamical system.

\begin{theorem}
\label{thm:Appr-wgtDOS}
Let $G$ and $H$ be lcsc groups where $G$ is additionally unimodular. Assume further that $\Omega\in\IDelGH$ is uniquely ergodic with measure $\mu$ and that $\win\in\Cc_c(H)$ is a continuous window function. Suppose $(\Omega_l,\mu_l,G)$
are Delone dynamical systems such that $\Omega_l \to \Omega$ 
in the Chabauty-Fell topology. 
Then, for every self-adjoint random bounded operator $A$, we have
$$
\mbox{w-}*\mbox{-}\lim_{l \to \infty} \eta^A_{\Omega_l,\win,\mu_l} = \eta^A_{\Omega,\win}
$$
in the weak-$*$-topology on $\mathcal{M}(\RR)$.
\end{theorem} 

\begin{proof}
Let $A$ be a random bounded operator. Since $\Omega$ is uniquely ergodic and $\Omega_l\to\Omega$ in the Chabauty-Fell topology, Theorem~\ref{thm:mainapprox} implies the convergence of the measures $(\mu_l)$ to $\mu$ in the weak-$\ast$ topology. We write $\eta_l = \eta^A_{\Omega_l,\win,\mu_l}$ and $\eta:=\eta^A_{\Omega,\win}$. Let $\phi\in\Cc_c(\RR)$. Then $B:=\phi(A)$ defined by the functional calculus is a random bounded operator. Thus, 
$$
\lim_{l\to\infty} \eta_l(\phi) 
	= \lim_{l\to\infty} 
		\int_{\Delr} 
			\sum_{x \in D} \rho(x) \langle \delta_x, B_D\delta_x \rangle_{D,\win}
		\,d\mu_l(D)
	= \int_{\Delr} 
			\sum_{x \in D} \rho(x) \langle \delta_x, B_D\delta_x \rangle_{D,\win}
		\,d\mu(D)
$$
follows since $(\mu_l)$ converges in the weak-$\ast$ topology to $\mu$ and $D\mapsto\sum_{x \in D} \rho(x) \langle \delta_x, B_D\delta_x \rangle_{D,\win}$ is continuous by Proposition~\ref{prop:RanOpProp}~(i). The latter expression is nothing but $\eta(\phi)$. Consequently, $\mbox{w-}\ast\mbox{-}\lim_{l\to\infty}\eta_l=\eta$ is derived.
\end{proof}

\medskip

The following corollary deals with the case that $H=\{e_H\}$ is the trivial group with continuous window function $\win= 1$.

\begin{cor}[Continuity of the density of states for Delone sets]
\label{cor:Appr-DOS}
Let $G$ be a unimodular lcsc group and $\Delr:=\Del$ be the compact space of Delone sets in $G$ with fixed $U$ and $K$. Consider a uniquely ergodic Delone dynamical system $\Omega\in\IDelG$. 
If $(\Omega_l,\mu_l,G)$
are Delone dynamical systems in $\IDelG$ such that $\Omega_l \to \Omega$ 
in the Chabauty-Fell topology, then, for every self-adjoint random bounded operator $A$,
$$
\mbox{w-}*\mbox{-}\lim_{l \to \infty} \eta^A_{\Omega_l,\mu_l} = \eta^A_{\Omega,\mu}
$$
in the weak-$*$-topology on $\mathcal{M}(\RR)$.
\end{cor}

\begin{remark} \label{rem:convergenceofspectra}
(i) As discussed in Remark~\ref{rem:groupoidCalg}, a random bounded operator on a dynamical system can be represented by an element of a groupoid $C^\ast$-algebra induced by the corresponding transversal. Since the isomorphism preserves the spectrum, the spectrum of both elements coincide. Due to Theorem~\ref{thm:TransHullHomeo}, the transversals $(\Tt_l)$ of a sequence of Delone dynamical systems $(\Omega_l,\mu_l,G)$ converge to the transversal $\Tt$ of $(\Omega,\mu,G)$ if and only the Delone dynamical systems converge. Thus, if the action of $G$ on $\Delr$ is amenable, then the convergence of the spectra
$$
\lim_{l\to\infty}\sigma(A_{\Omega_l}) = \sigma(A_\Omega)
$$
in the Hausdorff metric on $\RR$ holds additionally in Theorem~\ref{thm:Appr-wgtDOS} and Corollary~\ref{cor:Appr-DOS} by applying \cite[Theorem~3]{BeBeNi17}.

(ii) It is well-known that symbolic dynamical systems can be encoded as Delone sets \cite{LaPl03}. Hence, our result also applies to those systems. In \cite{BeckusThesis16,BeBeNi17}, several explicit constructions are given for periodic approximations of symbolic dynamical systems. It is proven there that the spectrum converges in the Hausdorff metric. With the convergence of the abstract density of states, Corollary~\ref{cor:Appr-DOS} complements these results via a measured quantity.
\end{remark}


\section{Approximations for cut-and-project schemes} 
\label{sec:CP-schemes}

In the upcoming section, we study the continuity of the autocorrelation and the density of states for the special class of regular models. It is well-known that the arising Delone dynamical systems are uniquely ergodic which enables us to apply the previously developed theory. The approximation is via weighted Delone sets with continuous window functions. Since regular model sets do not belong to that latter class, there are two involved limit processes. While the approximation of the Delone dynamical systems is guaranteed by lattice convergence, another limit takes care of the approximation of the regular window function. Our result includes approximations via periodic structures in the Euclidean situation, where one can choose sequences of lattices of the form $Q \ZZ^d$, where all entries of the $d$-dimensional matrix $Q$ are rational.  

\medskip

The relevant notions and properties of regular model sets and their dynamical systems are discussed in Subsection~\ref{ssec:CP-set}. The convergence of the autocorrelation is provided in Theorem~\ref{thm:Appr-ModAC} of Subsection~\ref{ssec:Appr-CP-AC}. For the class of strongly pattern equivariant operators, the density of states can be described in terms of the $G$-dynamics on the quotient space obtained from the projection lattice.
This allows us to the settle the convergence of the density of states for strongly pattern equivariant Schr\"odinger operators in Subsection~\ref{ssec:Appr-DOS-PEop}, c.f. Theorem~\ref{thm:Appr-ModDOS}.

\subsection{Regular model sets} 
\label{ssec:CP-set}

We now turn to the description of so-called {\em cut-and-project schemes}. Examples for model sets exist in abundance. Concerning the abelian 
and in particular the Euclidean situation, we refer to \cite{BaakeGrimm13} and references therein. 
A list of examples in non-abelian groups can be found in 
\cite{BjHaPo16}. We extensively use a parametrization map introduced in the abelian case by \cite{Sch00} (``torus parametrization'') and in the non-abelian world by \cite{BjHaPo16}. We follow the lines of \cite{BjHaPo16} to introduce the basic notations and concepts.

\medskip

A lattice $\Gamma < G \times H$ is {\em regular} if the restriction of $\pi_G$ to $\Gamma$ is injective and $\pi_H(\Gamma)$ is densely contained in $H$. A compact set $W\subseteq H$ is
called a {\em regular window for $\Gamma$} if 
\begin{itemize}
\item $m_H(\partial W) = 0$,
\item $\partial\,W \cap \pi_H(\Gamma) = \emptyset$,
\item $hW = W \Leftrightarrow h = e_H$, and 
\item $\overline{\mathring{W}} = W$.
\end{itemize}
Here, $\partial W$ denotes the topological boundary of the set $W$.

\begin{definition}[Regular model sets]
\label{def:regModSet}
Let $\Gamma$ be a regular lattice in $G \times H$. For a regular 
window $W \subset H$ for $\Gamma$, we call the set 
$$
P := P\big(G,H,\Gamma, W\big) :=  \pi_G \big( \Gamma \cap (G \times W) \big)
$$
a {\em regular model set}. 
\end{definition}

\begin{remark}\label{rem:identification}
By definition, a regular model set is a subset of $G$. However, since the
restriction of $\pi_G$ to $\Gamma$ is injective by assumption, 
each regular model set $P$ is naturally identified as the set $\Lambda(P) := \Gamma \cap (G \times W)$. 
As such, a regular model set can be seen as a weighted model set with constant weight $1$ if one extends the definition of weighted model sets to possibly non-uniform lattices. 
For cocompact lattices,  regular model sets are naturally identified with weighted model sets (and weighted Delone sets) in the language of Definition~\ref{def:WghtModSet}.
\end{remark}

Regular model sets have very nice properties. First of all, they are uniformly discrete and even stronger, they belong to the class of {\em sets of finite local complexity (FLC-sets)}, i.e.\@ $P^{-1}P$ is closed and locally finite. This observation is standard in the abelian world and it is also easily verified in the non-commutative situation, see \@\cite[Proposition~2.13]{BjHa16}. If
$\Gamma$ is chosen to be cocompact, the empty set is not contained in the hull $\Hh^P$. Thus, by \cite[Proposition~4.4]{BjHa16}, this implies that $P$ must be relatively dense. 
Hence, regular model sets $P$ as just constructed are Delone sets. 

\medskip

One of the main results in the aforementioned paper is that for regular model sets, the translation dynamics can be completely described in terms of the action of $G$ on the homogeneous space $Y:= G \times H /\Gamma$. Since the projection $\pi_H(\Gamma)$ is dense in $H$, the action $G \curvearrowright Y$ is minimal (by the duality principle). Moreover, one can show that in the present situation, this implies that $G \curvearrowright Y$ is uniquely ergodic, cf.\@\cite[Lemma~3.7]{BjHaPo16}. As was shown in the following theorem, these properties transfer to the (punctured) hull.

\begin{theo}[cf.\@ \cite{BjHaPo16}, Theorem~1.1] \label{prop:toruspara}
Let $P = P(G,H,\Gamma, W)$ be a regular model set. Then,
there exists a unique $G$-invariant probability  measure $\mu$  (which is also the unique $G$-stationary measure) on the punctured hull $\Hh^{P \, \times}= \Hh^P \setminus \{ \emptyset \}$ such
that $G \curvearrowright (\Hh^{P \, \times},\mu)$ is measurably isomorphic to $G \curvearrowright 
(Y,m_Y)$. In particular, as unitary $G$-representations, we have
$$
L^2(\Hh^{P \, \times}, \mu) \cong L^2(Y,m_Y).
$$
\end{theo}

In the case of a uniform lattice $\Gamma$ (and this is the case we are interested in), the hull and the punctured hull of a model set 
coincide. 
The proof of the theorem goes along a well-behaved map $\beta: \Hh^{P\, \times} \to Y$,
called {\em parametrization map}. More precisely, $\beta$ is $G$-equivariant,
has a closed graph (cf.\@ \cite[Theorem~3.1]{BjHaPo16}),
and the unique invariant measure $\mu$ is given by the relation $\beta_* \mu = m_Y$,
cf.\@ \cite[Theorem~3.4]{BjHaPo16}. The subset of the hull where $\beta$ is one-to-one is
conull with respect to $\mu$, see further elaborations below. Sticking to the notation in later discussed work, we 
denote 
$$
Y^{\operatorname{ns}} 
	:= \big\{ 
		(g,h)\Gamma 
		\,|\, 
		\partial\big(h^{-1}W_{\win}\big) \cap \pi_H(\Gamma) = \emptyset 
	\big\},
$$
and call this set the {\em non-singular points} of $Y$. The corresponding set of
non-singular points on the hull is written as $\Hh^{\operatorname{ns}}:= \beta^{-1}(Y^{\operatorname{ns}})$. It has been proven in \cite[Section~3]{BjHaPo16} that the following assertions hold, see also \cite{Sch00} for the abelian case.

\begin{proposition}[\cite{BjHaPo16}]
\label{prop:propertiesofXns}
Let $P := P\big(G,H,\Gamma, W\big)$ be a regular model set.
\begin{enumerate}[(i)]
\item The restriction $\beta_{|\Hh^{\operatorname{ns}}}: \Hh^{\operatorname{ns}} \to 
Y^{\operatorname{ns}}$ is bijective.
\item $\Hh^{\operatorname{ns}}$ is a set of full $\mu$-measure in $\Hh^{P\, \times}$, and $Y^{\operatorname{ns}}$ is a set of full 
$m_Y$-measure in $Y$.
\item Let $\mathcal{F}$ be a fundamental domain for $\Gamma$ in $G \times H$. 
For each $Q \in \Hh^{\operatorname{ns}}$, there is a unique choice $(g_Q, h_Q) \in \mathcal{F}$ with $\beta(Q)= (g_Q, h_Q)\Gamma$
and such that 
$$
Q 
	= g_{Q}\, \pi_G\big( \Gamma \cap (G \times h_{Q}^{-1}W) \big) 
	= \pi_G \big((g_Q, h_Q)\Gamma \cap (G \times W) \big).
$$
\item The dynamical system $G \curvearrowright 
Y$ is minimal. If in addition, $\Gamma$ is cocompact, then $G \curvearrowright \Hh^P$ is minimal as well.
\item If $\Gamma$ is uniform, then the parametrization map $\beta: \Hh^P \to Y$ is continuous. 
\end{enumerate}
\end{proposition}

As in previous chapters, we will exclusively deal with uniform lattices from this point on. When considering collections of lattices, we will assume that all elements in this collection are Delone with fixed parameters for uniform discreteness and relative denseness. More precisely, this means that they are contained in a set
$\Delr$ consisting of Delone sets in $G\times H$ with fixed parameters $U$ and $K$. The closed, $G\times H$-invariant subsets are denoted by $\IDelGH$. Whenever we only speak about Delone sets in $G$ we will specify.

\subsection{Approximation of the autocorrelation for regular model sets} 
\label{ssec:Appr-CP-AC}
Using Theorem~\ref{thm:Appr-wgtAC}, we obtain an approximation statement for the autocorrelation of regular model sets. For a cocompact lattice $\Gamma \in \Delr$, the homogeneous space $Y:=G\times H/\Gamma$ admits a unique $G\times H$-invariant, probability measure (the Haar measure) denoted by $m_Y$.

\medskip

Let $\win_l\in\Cc_c(H)$ and $W\subseteq H$ be a compact set. The characteristic function of the set $W$ is denoted by $\chi_W$. Furthermore, $W^\circ$ denotes the interior of the set $W$. We write $\win_l \searrow \chi_W$ if $\lim_{l\to\infty}\win_l(x)=\chi_W(x)$ and $\win_l(x)\geq \chi_W(x)$ for each $x\in H$. Analogously, $\win_l \nearrow \chi_{W^\circ}$ means $\lim_{l\to\infty}\win_l(x)=\chi_{W^\circ}(x)$ and $\win_l(x)\leq \chi_{W^\circ}(x)$ for each $x\in H$.

\begin{theorema}[Continuity of the autocorrelation for regular model sets] 
\label{thm:Appr-ModAC}
Let $P = P(G,H,\Gamma, W)$ be a regular model set.  Assume that $(\win_{l})$ is a sequence of continuous window functions such that  
$\win_l \searrow \chi_W$ pointwise in $H$. If $(\Gamma_n)_n$ are cocompact lattices such that $\Gamma_n \to \Gamma$ in $\Delr$, 
then 
$$
\mbox{w-}*\mbox{-}\lim_{l \to \infty} \lim_{n \to \infty} \gamma_{Y_n, \win_l} 
	= \mbox{w-}*\mbox{-}\lim_{l \to \infty} \gamma_{Y, \win_l}
	= \gamma_{P}
$$
in the weak-$*$-topology on $\mathcal{M}(G)$ where $\gamma_P:=\gamma_{\Hh^P}$ is the autocorrelation of the uniquely ergodic hull $\Hh^P$.
\end{theorema}

\begin{proof}
For the cocompact lattice, the homogeneous space $Y\subseteq\Delr$ is uniquely ergodic with measure $m_Y$. The convergence $\Gamma_n\to\Gamma$ implies $Y_n\to Y$ in the Chabauty-Fell topology by Proposition~\ref{prop:weakastlattice} and Theorem~\ref{thm:TransHullHomeo}. Hence, Theorem~\ref{thm:Appr-wgtAC} implies
$$
\mbox{w-}*\mbox{-}\lim_{n \to \infty} \gamma_{Y_n, \win_l}   = \gamma_{Y, \win_l} 
$$
for $l \in \NN$. According to Proposition~\ref{prop:ACweight}, we have that 
$$
\gamma_{Y, \win_l,m_Y} \big( f_1^{*} * f_2 \big) 
	= \int_{\Delr} \overline{\Pp_{\win_l}f_1(\Lambda)} \; \Pp_{\win_l}f_2(\Lambda)\,dm_Y(\Lambda)
$$
for all $f_1, f_2 \in \Cc_c(G)$. By Lemma~\ref{lem:PerProp}, the function $\Lambda\mapsto \overline{\Pp_{\win_l}f_1(\Lambda)}
\cdot \Pp_{\win_l}f_2(\Lambda)$ is continuous and converges pointwise to $\overline{\Pp_{\chi_W}(f_1)(\Lambda)} \cdot \Pp_{\chi_W}(f_2)(\Lambda)$. Furthermore, it is dominated by the measurable function $\Lambda\mapsto\Pp_{\win_1}(|f_1|)(\Lambda)\, \Pp_{\win_1}(|f_2|)(\Lambda)$. Hence, the dominated convergence theorem and Theorem~\ref{prop:toruspara} imply 
\begin{align*}
\lim_{l \to \infty} \gamma_{Y, \win_l,m_Y}\big( f_1^{*} * f_2 \big) 
	&= \int_Y \overline{\Pp_{\chi_W}(f_1)(\Lambda)} \; \Pp_{\chi_W}(f_2)(\Lambda) \; dm_Y(\Lambda)\\
	&= \int_{\Hh^P} \overline{\Pp_{\chi_W}(f_1)(\beta(D))} \; \Pp_{\chi_W}(f_2)(\beta(D)) \; d\mu(D),
\end{align*}
since $\beta:\Hh^P\to Y$ induces an isomorphism between $L^2(\Hh^P,\mu)$ and $L^2(Y,m_Y)$. For $D\in\Hh^{\operatorname{ns}}$, Proposition~\ref{prop:propertiesofXns}~(iii) implies that $D=\pi_G(\beta(D)\cap G\times W))$ while $\pi_G$ restricted to $\beta(D)\in Y^{\operatorname{ns}}$ is injective. With this at hand, a short computation yields
$$
\Pp_{\chi_W}(f)(\beta(D)) 
	= \sum_{x\in\beta(D)} f(\pi_G(x)) \chi_W(\pi_H(x)) 
	= \sum_{x\in\beta(D)\cap G\times W} f(\pi_G(x))
	= \sum_{y\in D} f(y)
	= \Pp f(D).
$$
Since $\Hh^{\operatorname{ns}}\subseteq \Hh^P$ has full measure by Proposition~\ref{prop:propertiesofXns}~(ii), the previous considerations imply
$$
\lim_{l \to \infty} \gamma_{Y, \win_l,m_Y}\big( f_1^{*} * f_2 \big) 
	= \int_{\Hh^P} \overline{\Pp f_1(D)} \; \Pp f_2(D) \; d\mu(D)
	= \gamma_P \big( f_1^{*} * f_2 \big).
$$
The dynamical system $G\curvearrowright\Hh^P$ is uniquely ergodic with measure $\mu$ and so $\gamma_P$ is the autocorrelation measure of $\Hh^P$. Hence, $\lim_{l \to \infty}\lim_{n \to \infty} \gamma_{Y_n, \win_l,m_n}=\gamma_P$ follows as the (weighted) autocorrelations are uniquely determined on the set $\{f_1^*\ast f_2\,|\, f_1,f_2\in\Cc_c(G)\}$.
\end{proof}

\begin{remark}
The assertion of Theorem~\ref{thm:Appr-ModAC} holds also if $\win_l \nearrow \chi_{W^\circ}$ pointwise in $H$ where $W^\circ$ is the interior of $W$. This follows by using Lebesgue's monotone convergence theorem instead of the dominated convergence theorem. Furthermore, one exploits that $\Lambda\cap\partial W=\emptyset$ for $\Lambda\in Y^{\operatorname{ns}}$ where $Y^{\operatorname{ns}}\subseteq Y$ has full measure. Consequently, $\Pp_{\chi_{W^\circ}}(f)=\Pp_{\chi_W}(f)$ follows for every $f\in\Cc_c(G)$ on a full measure set.
\end{remark}

\subsection{Approximation of the density of states of strongly pattern equivariant operators for regular model sets} 
\label{ssec:Appr-DOS-PEop}

We consider strongly pattern equivariant Schr\"odinger operators arising from strongly pattern equivariant functions as introduced in \cite{KePu00,Kel03,Kel08}. 
This class fits in the framework of random bounded operators considered
above. In fact, it forms a $\ast$-subalgebra $\PE\subseteq\KfinG$, where $\Rr$ is a finite subset of $P^{-1}P$ modeling the range of the operator. For a continuous window function $\win$ and a lattice $\Gamma\in\Delr$, every kernel $a\in\PE$ determines a random bounded operator $A$ as defined in Definition~\ref{def:RanOp}. We will show below that in this way, we obtain strongly pattern equivariant Schr\"odinger operators which can be represented by coefficients determined by strongly pattern equivariant functions. This class of operators can be seen as operators on the strip in $G\times H$ with continuous window functions $\win$. Specifically, the kernel $a\in\PE$ is lifted to $a^\win$ with associated random bounded operator $A^\win$, c.f. \@ Definition~\ref{def:LiftGen}. A detailed description is given below. This is the nice feature of strongly pattern equivariant Schr\"odinger operators that their spectral properties are compatible with the geometry induced by the cut and project scheme:
while from the topological stand point, the $G$-actions on the hull and on the quotient space $G \times H / \Gamma$ are different, the relevant measured information is preserved. 
Combining this with specific combinatorial features of strongly pattern equivariant Schr\"odinger operators, we can describe their DOS in terms of the dynamics on the homogeneous type space. 
Since the parametrization map is not a homeomorphism, this cannot be expected at all for general
random, bounded operators.
 

\medskip

The goal of this section is to prove the following theorem. 

\begin{theorema}[Continuity of the DOS for regular model sets] 
\label{thm:Appr-ModDOS}
Let $G$ and $H$ be lcsc groups where $G$ is unimodular. Let $P=P(G,H,\Gamma, W)$ be a regular model set. 
Assume that $(\win_{l})$ is a sequence of continuous window functions such that  
$\win_l \searrow \chi_W$ pointwise in $H$. 
Suppose $(\Gamma_n)_n$ are cocompact lattices satisfying $\Gamma_n \to \Gamma$ in $\Delr$. 
Then, for every self-adjoint kernel $a\in\PE$ with associated operator family $A$, we have
$$
\mbox{w-}*\mbox{-}\lim_{l \to \infty} \lim_{n \to \infty} \eta^{A^{\win_l}}_{Y_n, \win_l} 
	= \mbox{w-}*\mbox{-}\lim_{l \to \infty} \eta^{A^{\win_l}}_{Y, \win_l}
	= \eta^A_{\Hh^P}
$$
in the weak-$*$-topology on $\mathcal{M}(\RR)$.
\end{theorema}

We first need to introduce the notion of strongly pattern equivariant functions for Delone sets of finite local complexity. We emphasize at this point that all those Delone sets are subsets of $G$. For convenience, we use the symbol $\DelrG$ for all Delone sets in $G$ with fixed parameters $U$ and $K$. 

\medskip

A Delone set $P\in\DelrG$ is called of {\em finite local complexity} if $P^{-1}P$ is closed and discrete. Such a Delone set satisfies $D^{-1}D\subseteq P^{-1}P$ for all elements $D$ in the hull $\Hh^P$ and for every compact subset $F\subseteq G$ a constant $C>0$ such that $\sharp P\cap gK \leq C$ for all $g\in G$ \cite[Appendix~A]{BjHaPo16}. Moreover, the set $\{x^{-1}P\cap F\,|\, x\in P\}$ is finite for every compact $F\subseteq G$.

\begin{definition}
\label{def:PE}
Let $\Hh^P$ be the hull of $P$ with finite local complexity and $\Tt^P:=\{D\in\Hh^P\,|\, e\in D\}$ be its transversal. A function $f:\Tt^P\to\CC$ is called {\em strongly pattern equivariant} if there is an open, relatively compact subset $F\subseteq G$ such that $f(D)=f(D')$ if $D\cap F = D'\cap F$. 
\end{definition}

The following lemma provides a representation of the strongly pattern equivariant functions on $\Tt^P$.

\begin{lemma}
\label{lem:PE}
Let $P\in\DelrG$ be of finite local complexity. A function $f:\Tt^P\to\CC$ is strongly pattern equivariant if and only if there exist an $N\in\NN$ and an open, relatively compact set $F$ satisfying that for every $1\leq j\leq N$ there is a coefficient $p_j\in\CC$ and an element $x_j\in P$ such that
\begin{equation}\label{eq:PE}
f(D) = \sum_{j=1}^N p_j \; \prod_{\gamma\in x_j^{-1}P\cap F} \chi_D(\gamma)
\end{equation}
where $\chi_D(\gamma)=1$ iff $\gamma\in D$ and otherwise $\chi_D(\gamma)=0$. Furthermore, every strongly pattern equivariant function is continuous.
\end{lemma}

\begin{proof}
Recall that $\{x^{-1}P\cap F'\,|\, x\in P\}$ is finite for every compact set $F'$. Thus, $D_n\to D$ in the Chabauty-Fell topology of $\Tt^P$ implies that $D_n\cap F'=D\cap F'$ for $n$ large enough. Consequently, functions of the form (\ref{eq:PE}) are continuous and strongly pattern equivariant.

Let $f$ be strongly pattern equivariant and $F\subseteq G$ be the corresponding open, relatively compact subset such that $f(D)=f(D')$ whenever $D\cap F = D'\cap F$. Using finite local complexity, there are $x_1,\ldots, x_N\in P$ such that $\{x_1^{-1}P\cap F,\ldots,x_N^{-1}P\cap F\}=\{D\cap F\,|\, D\in\Tt^P\}$. Note that the product $\prod_{\gamma\in x_j^{-1}P\cap F} \chi_D(\gamma)$ is equal to one if and only if $x_j^{-1}P\cap F\subseteq D\cap F$. 

In general, it might happen that $x_j^{-1}P\cap F\subsetneq x_k^{-1}P\cap F$, namely the patch $x_k^{-1}P\cap F$ contains another patch $x_j^{-1}P\cap F$ appearing in $P$. If this is not the case we set $p_j:=f(x_j^{-1}P)$. Otherwise, we need to define the coefficients with more care. In order to do so, let $I(n)$ be the set of $k\in\{1,\ldots, N\}$ such that the patch $x_k^{-1}P\cap F$ contains exactly $n$ other patches appearing in $P$.  More precisely, for $n\in\NN_0:=\NN\cup\{0\}$, $I(n)\subseteq\{1,\ldots, N\}$ is the set of $k\in\{1,\ldots, N\}$ satisfying that there are exactly $n$ elements $x_j$ satisfying that $x_j^{-1}P\cap F\subsetneq x_k^{-1}P\cap F$. Clearly, $\bigsqcup_{n\in\NN_0}I(n)=\{1,\ldots, N\}$ is a disjoint union and $I(n)$ is empty if $n\geq N$. Furthermore, define 
$$
C(j)
	:=\big\{
		k\in\{1,\ldots, N\} 
		\,\big|\, 
		x_k^{-1}P\cap F\subsetneq x_j^{-1}P\cap F 
	\big\}
	\,,\qquad
	1\leq j\leq N\,,
$$
The set $C(j)$ is the collection of subpatches of the patch $x_j^{-1}P\cap F$. By iteration from $n=0$ up to (maximal) $N$, define $p_j:=f(x_j^{-1}P)-\sum_{k\in C(j)} p_k$ for $j\in I(n)$. Note that $C(j)$ is empty for $n=0$ and $j\in I(n)$. Hence, $p_j:=f(x_j^{-1}P)$. With this,  the coefficients $p_j$ are iteratively defined. It is left to show that with this choice of coefficients $p_j$, equation~\eqref{eq:PE} holds for all $D\in\Tt^P$.

Let $D\in\Tt^P$ with $D\cap F=x_L^{-1}P\cap F$ for some $L\in I(n_0)$. The disjointness of the $I(m)$ and the definition of the coefficient $p_L$ yield
\begin{align*}
\sum_{j=1}^N p_j \; \prod_{\gamma\in x_j^{-1}P\cap F} \chi_D(\gamma) 
	= \left( f(x_L^{-1}P) - \sum_{k\in C(L)} p_k \right)
		+ \sum_{m=0}^{N}\; \sum_{j\in I(m)\setminus\{L\}} p_j 
		\prod_{\gamma\in x_j^{-1}P\cap F} \chi_D(\gamma)\,.
\end{align*}
The latter sum is equal to $\sum_{k\in C(L)} p_k$ as the product $\prod_{\gamma\in x_j^{-1}P\cap F} \chi_D(\gamma)$ is one if and only if $x_j^{-1}P\cap F$ is a subpatch of $x_L^{-1}P\cap F$. Thus,
\begin{align*}
\sum_{j=1}^N p_j \; \prod_{\gamma\in x_j^{-1}P\cap F} \chi_D(\gamma) 
	= f(x_L^{-1}P) = f(D)
\end{align*}
follows by using that $f$ is strongly pattern equivariant and $D\cap F=x_L^{-1}P\cap F$.
\end{proof}

\begin{remark}
By exploiting the topology on the transversal in more detail, it is not difficult to show the converse of the previous lemma, i.e.\@ $f:\Hh^P\to\CC$ is strongly pattern equivariant if and only if $f$ is continuous and takes finitely many values, see (in a special case) \cite[Proposition~3.7.2]{BeckusThesis16}. 
\end{remark}

\begin{definition}
\label{def:PEop}
Let $P\in\DelrG$ be of finite local complexity. A collection of operators $A:=\big(A_D:\ell^2(D)\to\ell^2(D)\big)_{D\in\Hh^P}$ is called {\em strongly pattern equivariant Schr\"odinger operator} if each $A_D$ acts as
$$
A_Du (x) 
	= \left(
		\sum_{\gamma\in\Rr}  
			q_\gamma\big(x^{-1}D\big)\, \chi_D(x\gamma^{-1})\, u(x\gamma^{-1}) 
			+ 
			\overline{q_\gamma\big((x\gamma)^{-1}D\big)}\, \chi_D(x\gamma)\, u(x\gamma)
	\right) 
		+ V(x^{-1}D) u(x)
	\,.
$$
where $\Rr\subseteq P^{-1}P$ is finite and the maps $V:\Tt^D\to\RR$ and $q_\gamma:\Tt^D\to\CC,\, \gamma\in\Rr,$ are strongly pattern equivariant. The finite set $\Rr$ is called {\em range of the operator family $A$}.
\end{definition}

Next, suitable kernels are introduced for Delone sets of finite local complexity. First, we show that these kernels generate all finite range Schr\"odinger operators with strongly pattern equivariant potentials, c.f. Proposition~\ref{prop:PEop}. These operators are generalizations of strongly pattern equivariant Schr\"odinger operators on the Cayley graph of a countable group, c.f.\@ e.g.\@ \cite[Theorem~3.7]{BeckusThesis16}. The kernels are defined by functions $\vartheta_\gamma\in\Cc_c(G)$ and hence, they can be lifted to $G\times H$. In this way, we can lift all strongly pattern equivariant Schr\"odinger operators, c.f. Definition~\ref{def:LiftGen}. 

\medskip

For  $P\in\DelrG$ of finite local complexity and $\gamma\in P^{-1}P$, we can fix a function $\vartheta_\gamma\in\Cc_c(G)$ such that $\supp(\vartheta_\gamma)\cap P^{-1}P=\{\gamma\}$ and $\vartheta_\gamma(\gamma)=1$. The existence of $\vartheta_\gamma\in\Cc_c(G)$ follows by Urysohn's lemma and that $P^{-1}P$ is discrete since $P$ is of finite local complexity. Then
$$
s_\gamma(x,D,y):=\vartheta_\gamma(x^{-1} y)
$$ 
defines a kernel of finite type on $\DelrG$. It is straight forward to check that, $s_\gamma$ on $\DelrG$ is a kernel of finite type according to Definition~\ref{def:kernel}. 

\medskip

For $\Rr\subseteq P^{-1}P$ finite, define $\PE$ as the $\ast$-subalgebra of $\KfinG$ generated by $\{s_\gamma\,|\, \gamma\in\Rr\}$. The following statement shows that every strongly pattern equivariant Schr\"odinger operator $A=(A_D)_{D\in\Hh^P}$ arises from a self-adjoint element $a\in\PE$. This statement is a generalization of \cite[Theorem~3.7]{BeckusThesis16} which deals with Cayley graphs of countable groups.

\begin{proposition}
\label{prop:PEop}
Let $P\in\DelrG$ be of finite local complexity. A collection of operators $A:=\big(A_D:\ell^2(D)\to\ell^2(D)\big)_{D\in\Hh^P}$ is a strongly pattern equivariant Schr\"odinger operator if and only if there is a finite $\Rr\in P^{-1}P$ and a self-adjoint $a\in\PE$ such that $A_D=\lambda^D(a), D\in\Hh^P$.
In particular, every strongly pattern equivariant Schr\"odinger operator $A=(A_D)_{D\in\Hh^P}$ is the restriction of a self-adjoint random bounded operator $\tilde{A}=\big(\pi^D(a)\big)_{D\in\DelrG}$ of finite range to the hull $\Hh^P$.
\end{proposition}

The proof of the proposition relies on the following lemma.

\begin{lemma}
\label{lem:RepSgam}
Let $P\in\DelrG$ be of finite local complexity and $\Rr:=\{\gamma_1,\ldots,\gamma_N\}\subseteq P^{-1}P$. Define $p_{\gamma_j}:=s_{\gamma_j}^\ast\star s_{\gamma_j}$, $a:=p_{\gamma_1}\star\ldots\star p_{\gamma_N}$ and $b:=\sum_{\gamma\in\Rr} s_{\gamma} + s_{\gamma}^\ast$. Then the representations
$$
\lambda^D(a) u (x) 
	= \left(\prod_{\gamma\in\Rr} \chi_{x^{-1}D}(\gamma^{-1}) \right) u(x)
	,\;\;
\lambda^D(b) u (x) 
	= \sum_{\gamma\in\Rr} \chi_{x^{-1}D}(\gamma)\, u(x\gamma) + \chi_{x^{-1}D}(\gamma^{-1})\, u(x\gamma^{-1}),
$$
hold for all $u\in\ell^2(D)$, $x\in D$ and $D\in\Hh^P$.
\end{lemma}

\begin{proof}
Let $\gamma\in\Rr$. Since $D^{-1}D\subseteq P^{-1}P$ and $P^{-1}P\cap\supp(\vartheta_\gamma)=\{\gamma\}$, the value $\vartheta_\gamma(z^{-1} x)$ is one if $z^{-1} x=\gamma$ and otherwise zero for $z^{-1} x\in D^{-1}D$. Thus, a short computation leads to
$$
p_\gamma (x,D,y) 
	= \sum_{z\in D} \overline{s_\gamma(z,D,x)}\, s_\gamma(z,D,y) 
	= \sum_{z\in D} \vartheta_\gamma(z^{-1} x)\, \vartheta_\gamma(z^{-1} y)
	= \chi_D(x\gamma^{-1})\, \delta_x(y)
	.
$$
With this at hand, the corresponding operator defined by the representation $\lambda^D$ satisfies
$$
\lambda^D(p_\gamma) u (x) 
	= \sum_{y\in D} p_\gamma(x,D,y) \, u(y) 
	= \chi_D(x\gamma^{-1})\, u(x)
	.
$$
Similar calculations imply $\lambda^D(s_\gamma) u (x) = \chi_D(x\gamma)\, u(x\gamma)$ and $\lambda^D(s^\ast_\gamma) u (x) = \chi_D(x\gamma^{-1})\, u(x\gamma^{-1})$. Hence, the desired identities follow by using that $\lambda^D$ is a $\ast$-representation.
\end{proof}

\medskip

With this at hand, one obtains the validity of the above proposition.

\medskip

{\bf Proof of Proposition~\ref{prop:PEop}.}
We sketch the proof of the claim of the equivalence using Lemma~\ref{lem:PE} and Lemma~\ref{lem:RepSgam}. The detailed computations are straightforward and left to the reader. According to Lemma~\ref{lem:PE}, every strongly pattern equivariant functions is a finite linear combination of functions of the form $\prod_{\gamma\in x_j^{-1}P\cap F} \chi_D(\gamma)$. Combining this with Lemma~\ref{lem:RepSgam}, every multiplication operator $u(x)\mapsto V(x^{-1}D) u(x)$, with $V$ being a strongly pattern equivariant function, is given by $\lambda^D(a_V)$ where $a_V\in\PE$ is a finite linear combination of elements of the form $a:=p_{\gamma_1}\star\ldots\star p_{\gamma_N}$ and $p_{\gamma_j}:=s_{\gamma_j}^\ast\star s_{\gamma_j}$. Furthermore, the operator $u(x)\mapsto\chi_{x^{-1}D}(\gamma)\, u(x\gamma) + \chi_{x^{-1}D}(\gamma^{-1})\, u(x\gamma^{-1})$ is represented by $\lambda^D(b_\gamma)$ where $b_\gamma:=s_{\gamma} + s_{\gamma}^\ast$ for $\gamma\in\Rr$. Since every strongly pattern equivariant Schr\"odinger operator $A_D$ is a finite linear combination of compositions of such operators, an $a\in\PE$ can be constructed such that $A_D=\lambda^D(a)$. Conversely, every self-adjoint $a\in\PE$ is a finite linear combination of such elements and so it defines a strongly pattern equivariant Schr\"odinger operator. 

For the ``in particular'' statement, we argue as follows: Since $\PE$ is a $\ast$-subalgebra of $\KfinG$ and $a$ is self-adjoint, the random bounded operator $\tilde{A}:=\big(\pi^D(a)\big)_{D\in\DelrG}$ has finite range and is self-adjoint. Furthermore, $A$ is the restriction of $\tilde{A}$ to the hull $\Hh^P$.\hfill$\Box$

\medskip

Next, we define a suitable class of random bounded operators that can be lifted to $G\times H$ by lifting the corresponding generators.

\begin{definition}
\label{def:LiftGen}
Let $\win\in\Cc_c(H)$ be a continuous window function and $P\in\DelrG$ be of finite local complexity. Define the {\em lift of $s_\gamma$} by $s^\win_\gamma\in \Cc\big(X_{\supp(\win)}\big)$ by 
$$
s^\win_\gamma(x,\Lambda,y)
	:=\vartheta_\gamma\big(\pi_G(x^{-1} y)\big)\, \win(\pi_H(x))\, \win(\pi_H(y)).
$$
\end{definition}

With this at hand, $a^\win$ is the corresponding {\em lift of $a\in\PE$} to $\Kfin$ by lifting each of the generators. 
Furthermore, $A^\win$ denotes the associated operator family of $a^\win$. 

\medskip

Before proving Theorem~\ref{thm:Appr-ModDOS}, we need the following two technical lemmas.

\begin{lemma}
\label{lem:techn1}
Let $\win\in\Cc_c(H)$ be a continuous window function and $\gamma_1,\ldots,\gamma_N\in P^{-1}P$. Consider the kernel of finite range $a:=s^{\win}_{\gamma_1}\star\ldots\star s^{\win}_{\gamma_N}$. For $\Lambda\in\Delr$ and $y\in\Lambda$ with $\win(\pi_H(y))\neq 0$, the identity $\langle\delta_y,\, \lambda^\Lambda(a)\delta_y\rangle_{\Lambda,\win} = a(y,\Lambda,y)\, \win^2(\pi_H(y))$ holds and $a(y,\Lambda,y)$ equals to
$$
\sum_{y_1\in \Lambda} \ldots \sum_{y_{N-1}\in \Lambda}
	s^{\win}_{\gamma_1}(y,\Lambda,y_1)\, \win(\pi_H(y_1))\, 
	s^{\win}_{\gamma_2}(y_1,\Lambda,y_2)\, \win(\pi_H(y_2))
	\ldots 
	s^{\win}_{\gamma_N}(y_{N-1},\Lambda,y)\, \win(\pi_H(y_{N-1}))
	\,.
$$
\end{lemma}

\begin{proof}
The identity $\langle\delta_y,\, \lambda^\Lambda(a)\delta_y\rangle_{\Lambda,\win} = a(y,\Lambda,y)\, \win^2(\pi_H(y))$ follows by a short calculation. Furthermore, we have
$$
s^{\win}_{\gamma_1}\star s^{\win}_{\gamma_2}(y,\Lambda,y) 
	= \sum_{y_1\in\Lambda} s^{\win}_{\gamma_1}(y,\Lambda,y_1)\, s^{\win}_{\gamma_2}(y_1,\Lambda,y)\, \win(\pi_H(y_1))
	\,.
$$
With this, the desired identity follows by induction over $N\in\NN$.
\end{proof}

\begin{lemma}
\label{lem:techn2}
Let $P=P(G,H,\Gamma,W)$ be a regular model set and $\win_l\in\Cc_c(H)$ be a sequence of continuous window functions such that $\win_l\searrow\chi_W$ (or $\win_l\nearrow\chi_{W^\circ}$) pointwise in $H$. For $\gamma_1,\ldots,\gamma_N\in P^{-1}P$, consider $a_l:=s^{\win_l}_{\gamma_1}\star\ldots\star s^{\win_l}_{\gamma_N}$ and $a:=s_{\gamma_1}\star\ldots\star s_{\gamma_N}$. For $\Lambda\in Y^{\operatorname{ns}}$ and $D\in\Hh^{\operatorname{ns}}$ with $\beta(D)=\Lambda$, the identity 
$$
\lim_{l\to\infty} \sum_{y\in\Lambda} \rho(\pi_G(y))\, \langle\delta_y,\, \lambda^\Lambda(a_l)\delta_y\rangle_{\Lambda,\win_l}
	= \sum_{x\in D} \rho(x)\, \langle\delta_x,\, \lambda^D(a)\delta_x\rangle_{D}
$$
holds for any $\rho\in\Cc_c(G)$, where $\beta:\Hh^P\to Y$ is the parametrization map.
\end{lemma}

\begin{proof}
For each $l\in\NN$, we have 
$$
s^{\win_l}_{\gamma_i}(z,\Lambda,z')
	= \vartheta_{\gamma_i}\big(\pi_G(z)^{-1}\pi_G(z')\big)\, \win_l(\pi_H(z))\, \win_l(\pi_H(z'))
	\,,\qquad
	z,z'\in\Lambda
$$
and the projection $\pi_G:G\times H\to G$ is a group homomorphism. Thus, Lemma~\ref{lem:techn1} implies that $\lim_{l\to\infty}\langle\delta_y,\, \lambda^\Lambda(a_l)\delta_y\rangle_{\Lambda,\win_l}$ equals to
$$
\chi_W(\pi_H(y)) \sum_{y_1\in \Lambda\cap G\times W} \ldots \sum_{y_{N-1}\in \Lambda\cap G\times W}
	\vartheta_{\gamma_1}\left(\pi_G(y^{-1}y_1)\right)	
	\vartheta_{\gamma_2}\left(\pi_G(y_1^{-1}y_2)\right)
	\ldots
	\vartheta_{\gamma_N}\left(\pi_G(y_{N-1}^{-1}y)\right)
$$
by the pointwise convergence $\win_l\searrow\chi_W$ and as all involved sums are finite. Since $\beta(D)=\Lambda$ and $P$ is a regular model set, Proposition~\ref{prop:propertiesofXns}~(iii) yields $\pi_G(\Lambda\cap G\times W)=D$ while $\pi_G:\Lambda\cap G\times W \to D$ is injective. Hence, $\lim_{l\to\infty} \sum_{y\in\Lambda} \rho(\pi_G(y))\, \langle\delta_y,\, \lambda^\Lambda(a_l)\delta_y\rangle_{\Lambda,\win_l}$ is equal to
$$
\sum_{x\in D} \rho(x) 
	\sum_{x_1\in D} \ldots \sum_{x_{N-1}\in D}
		\vartheta_{\gamma_1}\left(x^{-1}x_1\right)	
		\vartheta_{\gamma_2}\left(x_1^{-1}x_2\right)	
		\ldots
		\vartheta_{\gamma_N}\left(x_{N-1}^{-1}x\right)	
	= \sum_{x\in D} \rho(x)\, \langle\delta_x,\, \lambda^D(a)\delta_x\rangle_{D}
$$
by using Lemma~\ref{lem:techn1} again for the trivial group $H=\{e_H\}$ with $\win=1$.
\end{proof}

\medskip

We are now ready to prove the main theorem of this subsection.

\medskip

{\bf Proof of Theorem~\ref{thm:Appr-ModDOS}.}
For the cocompact lattice, the homogeneous space $Y\in\IDelGH$ is uniquely ergodic with measure $m_Y$. The convergence $\Gamma_n\to\Gamma$ implies $Y_n\to Y$ in the Chabauty-Fell topology by Proposition~\ref{prop:weakastlattice} and Theorem~\ref{thm:TransHullHomeo}. Hence, for fixed $l \in \NN$, Theorem~\ref{thm:Appr-wgtDOS} implies
$$
\mbox{w-}*\mbox{-}\lim_{n \to \infty} \eta^{A^{\win_l}}_{Y_n, \win_l}   = \eta^{A^{\win_l}}_{Y, \win_l}.
$$
for the strongly pattern equivariant Schr\"odinger operators of finite range $A^{\win_l}$ induced by the lift $a^{\win_l}$, see Definition~\ref{def:LiftGen}.

For $\phi\in\Cc(\RR)$, $\phi(A^{\win_l})$ is defined by the functional calculus and can be approximated by polynomials of $A^{\win_l}$. Hence, it suffices to show $\lim_{l \to \infty} \eta^{A^{\win_l}}_{Y, \win_l}(p) = \eta^A_{\Hh^P}(p)$ for every polynomial $p$. By definition, $p(a)$ for any polynomial $p$ and $a\in\PE$ is a finite sum of monomials in $\{s_\gamma\,|\,\gamma\in\Rr\}$ and their adjoint and similarly for $a^{\win_l}$. Thus, it suffices by linearity to show that $\lim_{l \to \infty} \eta^{B^{\win_l}}_{Y, \win_l}(\operatorname{Id}) = \eta^B_{\Hh^P}(\operatorname{Id})$ where $B$ is an operator family induced by such a monomial and $B^{\win_l}$ is the corresponding lift. Since $s_\gamma^\ast = \vartheta_{\gamma^{-1}}\circ^{-1}$, there is no loss of generality in proving the desired identity only for monomials in $\{s_\gamma\,|\,\gamma\in\Rr\}$. So let $b:= s_{\gamma_1}\star\ldots\star s_{\gamma_N}$ where $\gamma_i\in\Rr$ and $b_l:=s^{\win_l}_{\gamma_1}\star\ldots\star s^{\win_l}_{\gamma_N}$ be the corresponding lift. Furthermore, set $B_\Lambda^{\win_l}:=\lambda^\Lambda(b_l)$ for $\Lambda\in Y$ and $B_D:=\lambda^D(b)$ for $D\in\Hh^P$.

The map $f_{\win_l}:Y\to\CC$ defined by
$
f_{\win_l}(\Lambda) \;
	:= \; \sum_{y\in\Lambda} \rho\big(\pi_G(y)\big) \; \langle \delta_y, B_\Lambda^{\win_l} \delta_y\rangle_{\Lambda,\win_l}
$
is continuous by Proposition~\ref{prop:RanOpProp} and 
$$
0\leq |f_{\win_l}(\Lambda)|
	\leq \sum_{y\in\Lambda} \rho\big(\pi_G(y)\big) \; 
		\langle \delta_y, \lambda^\Lambda(|b^{\win_1}|) 
		\delta_y\rangle_{\Lambda,\win_1}
$$ 
for all $\Lambda\in Y$ and $l\in\NN$. The latter function in $\Lambda$ is continuous and hence integrable. Consequently, the dominated convergence theorem implies
$$
\lim_{l\to\infty} \eta^{B^{\win_l}}_{Y, \win_l}(\operatorname{Id})
	= \int_Y \lim_{l\to\infty} \sum_{y\in\Lambda} \rho(\pi_G(y))\, \langle\delta_y,\, \lambda^\Lambda(b_l)\delta_y\rangle_{\Lambda,\win_l} \, dm_Y(\Lambda)
	\,.
$$
With this at hand, Theorem~\ref{prop:toruspara} and Proposition~\ref{prop:propertiesofXns} yield
$$
\lim_{l\to\infty} \eta^{B^{\win_l}}_{Y, \win_l}(\operatorname{Id})
	= \int_{\Hh^{\operatorname{ns}}} \lim_{l\to\infty} \sum_{y\in\beta(D)} \rho(\pi_G(y))\, \langle\delta_y,\, \lambda^{\beta(D)}(b_l)\delta_y\rangle_{\beta(D),\win_l} \, d\mu(D)
	\,.
$$
Using Lemma~\ref{lem:techn2}, $\lim_{l \to \infty} \eta^{B^{\win_l}}_{Y, \win_l}(\operatorname{Id}) = \eta^B_{\Hh^P}(\operatorname{Id})$ is deduced as $\Hh^{\operatorname{ns}}\subseteq\Hh^P$ has full measure (Proposition~\ref{prop:propertiesofXns}~(ii)).\hfill$\Box$

\begin{remark}
The assertion of Theorem~\ref{thm:Appr-ModDOS} holds also if $\win_l \nearrow \chi_{W^\circ}$ pointwise in $H$ where $W^\circ$ is the interior of $W$. This follows by using Lebesgue's monotone convergence theorem instead of the dominated convergence theorem. Furthermore, one exploits that $\Lambda\cap\partial W=\emptyset$ for $\Lambda\in Y^{\operatorname{ns}}$ where $Y^{\operatorname{ns}}\subseteq Y$ has full measure. 
\end{remark}


\section{Appendix} \label{sec:appendix}

\subsection{Chabauty-Fell topology} 
\label{ssec:appendix-ChaFe}

For convenience of the reader, we give a proof for the well-known fact that translation actions on Chabauty-Fell spaces are jointly continuous. 

\begin{proposition}
\label{prop:ActClCon}
Let $G$ be a lcsc acting on a compact, metrizable space $X$ by homeomorphisms. Then, $G$ also
acts on $\cs(X) = \ks(X)$ via homeomorphisms, i.e\@ the action mapping 
$$
\alpha:G\times\cs(X)\to\cs(X)\,,\qquad 
	(g,\Ff)\mapsto \alpha(g,\Ff):= g\cdot \Ff := \{g\cdot x\, |\, x\in \Ff\}
$$
is jointly continuous.
\end{proposition}

\begin{proof}
Let $g\in G$ and $F\in\cs(X)$. Consider a Chabauty-Fell neighborhood $\us(K,\os)$ of $g\cdot F$ with $\os:=\{O_1,\ldots,O_n\}$. Then, for $i=1,\ldots,n$, there are $x_i\in F$ such that $gx_i\in g\cdot F\cap O_i$. Due to continuity of the group action, there is an open neighborhood $U_i$ of $g$ and an open neighborhood $O_i'$ of $x_i$ such that $U_i\cdot O_i'\subseteq O_i$.

Since $F\cap g^{-1}\cdot K=g\cdot F\cap K=\emptyset$ and $X$ is a normal space, there is an open set $V'$ containing $g^{-1}\cdot K$ such that $\overline{V'}$ is compact and $\overline{V'}\cap F=\emptyset$. For each $x\in K$, there are open sets $V_x\ni g^{-1}$ and $V'_x\ni x$ satisfying $V_x\cdot V'_x\subseteq V'$ by continuity of the group action and as $g^{-1}x\in V'$. By compactness of $K$, there are $x_1,\ldots, x_m\in K$ with $V_l:=V_{x_l}$ and $V'_l:=V'_{x_l}$ satisfying $K\subseteq \bigcup_{l=1}^m V'_l$. Note that $V_l^{-1}$ is an open neighborhood of $g$.

The finite intersection $V:=\bigcap_{i=1}^n U_i \,\cap\, \bigcap_{l=1}^m V_l^{-1}$ is an open neighborhood of $g$. Further, define the Chabauty-Fell open set $\us:=\us\big(\overline{V'},\{O'_1,\ldots, O'_n\}\big)$. By construction, $O'_i$ intersects $F$ for every $i=1,\ldots, n$ and $\overline{V'}\cap F=\emptyset$. Hence, $\us$ is an Chabauty-Fell open neighborhood of $F$. We show that $V\cdot\us=\alpha(V,\us)\subseteq\us(K,\os)$: 

Let $h\in V$ and $F'\in\us$. Since $F'\in\us$, there exists an $y_i\in F'\cap O'_i\neq\emptyset$ and by the previous considerations $hy_i\in V\cdot O'_i\subseteq U_i\cdot O'_i\subseteq O_i$ holds. Hence, $h\cdot F'\cap O_i\neq\emptyset$ for $i=1,\ldots,n$. Furthermore, we derive
$$
 F'\cap h^{-1}\cdot K
	\subseteq F'\cap \left( h^{-1}\cdot \bigcup_{l=1}^m V'_l\right)
	\subseteq F'\cap \left(\bigcup_{l=1}^m h^{-1}\cdot V'_l\right)
	\subseteq F'\cap \overline{V'}
	=\emptyset
$$
by using that $h^{-1}\in V_l$ and the inclusion $V_l\cdot V'_l\subseteq V'$ holds. Hence, $h\cdot F' \cap K = \emptyset$ follows. Altogether, $h\cdot F'$ is contained in $\us(K,\os)$.
\end{proof}

\subsection{Groupoids, von-Neumann algebras and the density of states measure} 
\label{ssec:appendix-Groupoid}

Based on the non-commutative integration theory of A. Connes \cite{Con78}, the (integrated) density of states is described in \cite{LeSt03-Algebras,LeSt03-Delone,LePeVe07} in the setting of von Neumann algebras arising by a groupoid structure. In the following, we recall this setting in the specific case of Delone dynamical systems on lcsc groups.

\medskip

Let $G$ and $H$ be lcsc groups where $G$ is unimodular. Consider the compact space of Delone sets $\Delr:=\Del$ in $G\times H$ with some fixed parameters $U,K$. The group $G$ acts continuously on $\Delr$ by translation $gD:=(e_H,g)D$. Suppose $\Omega\in\IDelGH$ is equipped with a $(G\times H)$-invariant measure $\mu$ on $\Omega$. Furthermore, let $\win\in\Cc_c(H)$ be a (continuous) window function. We introduce a von Neumann algebra $\Nn(\Omega,G\times H,\mu,\win)$ of random bounded operators on $\Omega$. The purpose of this subsection is to introduce the density of states by showing

\begin{proposition}
\label{prop:GenTrace}
Let $G$ be  $\rho\in\Cc_c(G)$ be nonnegative satisfying $\int_G \rho\, dm_L=1$. The map $\tau_{\Omega,\win}:\Nn^+(\Omega,\mu,G,\win)\to [0,\infty)$ defined by
$$
\tau_{\Omega,\win}(A) 
	:= \int_{\Omega} \;
		\sum_{x \in D} 
			\rho(\pi_G(x))
			\,\big\langle 
				\delta_{x}, \phi\big( A_{D} \big) \delta_x
			\big\rangle_{D,\win}\, 
	d\mu(D),
$$
is faithful, normal, finite trace and it does not depend on the choice of $\rho$. Moreover, $\tau$ uniquely and continuously extends to $\Nn(\Omega,G\times H,\mu,\win)$.
\end{proposition}

With this at hand, the abstract density of states of a self-adjoint element of the von Neumann algebra can be defined.

\begin{definition}
Let $A\in\Nn(\Omega,G\times H,\mu,\win)$ be self-adjoint. Then the measure $\eta_{\Omega,\win,\mu}^A$ defined by $\eta_{\Omega,\win,\mu}^A(B):=\tau_{\Omega,\win}(\chi_B(A))$ for a Borel measurable set $B\subseteq\RR$ is called {\em abstract density of states}.
\end{definition}

The authors \cite{LeSt03-Delone} associated a von Neumann algebra with Delone dynamical systems on $\RR^d$ by defining a suitable groupoid $\gs$ and a $\gs$-space $(\Xx,\pi,J)$, c.f.\@ \cite[Definition~2.5]{LePeVe07}. This approach extends naturally to Delone dynamical systems of a unimodular lcsc group. For convenience of the reader, a guideline and the main concepts are described here. The proofs follow the same lines without any difficulty.

\medskip

A set $\gs$ equipped with a composition $\gs^{(2)}\subseteq\gs\times\gs\to\gs, (\gamma,\rho)\mapsto \gamma\rho,$ and  an inverse $\gs\to\gs, \gamma\to\gamma^{-1},$ is called {\em groupoid} if the composition is associative, $(\gamma,\gamma^{-1}), (\gamma^{-1},\gamma)\in\gs^{(2)}$ for all $\gamma\in\gs$ and for every $(\gamma,\rho)\in\gs^{(2)}$, we have $\gamma\rho\rho^{-1}=\gamma$ and $\gamma^{-1}\gamma\rho=\rho$. It is worth pointing out that the composition is only partially defined. Naturally associated with a groupoid is its {\em range map} $r:\gs\to\gs, r(\gamma):=\gamma\gamma^{-1},$ and {\em source map} $s:\gs\to\gs, r(\gamma):=\gamma^{-1}\gamma,$ as well as the {\em unit space} $\gs^{(0)}=r(\gs)=s(\gs)$. The set $\gs^x:=r^{-1}(x)$ is called {\em $r$-fiber} of $x\in\gs^{(0)}$. 

\medskip

The groupoids considered here are endowed with a locally compact, second countable, Hausdorff topology such that the corresponding composition and inverse (and hence, the range and source map) are continuous. Specifically, if the groupoid is endowed with the corresponding Borel-$\sigma$ algebra then all related maps are measurable. Furthermore, the $r$-fiber $\gs^x$ is a measurable subset of $\gs$.

\medskip

Define the topological groupoid $\gs:=\Omega\times G$ with composition $(D,g)\circ(g^{-1}D,h) := (D,gh)$ and inverse $(D,g)^{-1}:= (g^{-1}D,g^{-1})$. The range map is given by $r(D,g)=(D,e_G)$ and the source map by $s(D,g)=(g^{-1}D,e_G)$ and hence, the unit space $\gs^{(0)}$ is homeomorphic to $\Omega$. Define
$$
\Xx	:= \big\{ 
		(D,x)\in \Omega\times G\times H
		\,\big|\,
		x\in D,\,
		\win(\pi_H(x))\neq 0
	\big\}
	\subseteq
	\Omega\times G\times H
$$
equipped with the induced topology. Note that $\Xx$ is not closed in general due to the constraint that $\win(\pi_H(x))\neq 0$. Define the continuous projection $\pi:\Xx\to\Omega,\, \pi(D,x):=D,$ and the corresponding fibers $\Xx^D:=\pi^{-1}(D)\simeq \{x\in D\,|\,\win(\pi_H(x))\neq 0\}$. Furthermore, $\gs$ acts on $\Xx$ via $J$ defined by $J(D,g):\Xx^{s(D,g)}\to\Xx^{r(D,g)},\, (g^{-1}D,y)\mapsto (D,gy),$ satisfying $J(\gamma\eta)=J(\gamma)J(\eta)$ and $J(\gamma^{-1})=J(\gamma)^{-1}$. Hence, the triple $(\Xx,\pi,J)$ is a $\gs$-space according to \cite[Definition~2.5]{LePeVe07}. Denote by $\Ff(\Xx)$ the set of all measurable functions on $\Xx$ and $\Ff^+(\Xx)$ the subset of nonnegative measurable functions.
The map $\alpha:\Omega\to\Rr(\Xx),\, \alpha^D(f):= \sum_{x\in D} f(D,x)\, \win(\pi_H(x)),$ defines a random variable (in the sense of Connes) with values in $\Xx$, c.f. \cite[Definition~2.5]{LePeVe07}. The proof follows the same lines as \cite[Corollary~2.6]{LeSt03-Delone} by additionally using the $G$-invariance of $\win\circ\pi_H$ and that $(D,x)\in\Xx$ implies $\win(\pi_H(x))\neq 0$. 

\medskip

Furthermore, $\nu:\Omega\to\Rr(\gs)$ defined by $D\mapsto\nu^D(f):=\int_G f(D,h) \, dm_L(h)$ is a transverse function (left-continuous Haar system) of $\gs$ in terms of \cite[Definition~2.3]{LePeVe07}. In order to apply the theory developed in \cite{LePeVe07}, the groupoid structure with the corresponding measures needs to be an {\em admissible setting}, c.f. \cite[Definition~2.4]{LePeVe07}. Thus, the measure $\mu$ on $\Omega$ has to be $\nu$-invariant where $\nu$ is the transverse function. A short computation implies that this only holds if $G$ is unimodular, namely the Haar measure is left and right invariant. 

\medskip

Moreover, the existence of a strictly positive function $u_0\in\Ff^+(\Xx)$ needs to be guaranteed such that $\nu\ast u_0 (p) = 1$ for all $p\in\Xx$, c.f. \cite[Definition~2.6]{LePeVe07}. Such a function is defined by $u_0(D,x):= f(\pi_G(x))$, where $f:G\to(0,\infty)$ satisfies $\int_G f dm= 1$. Then the identity $\nu\ast u_0 (p) = 1$ follows for all $p\in\Xx$ by a short computation invoking that $G$ is unimodular.

\medskip

The following lemma assures that the trace is independent of the choice of the function $\rho$ in Proposition~\ref{prop:GenTrace}.

\begin{lemma}
\label{lem:RanVarInd}
Let $\beta:\Omega\to\Rr(\Xx)$ be random variable (in the sense of Connes) with values in $\Xx$. For every $F\in\Ff^+(\Xx)$, the integral $\int_\Omega \beta^D(F)\, d\mu(D)$ is independent of $F$ if,  for all $(D,x)\in\Xx$, we have $\int_G F(g^{-1}D,g^{-1}x)dm_L(g)=1$.
\end{lemma}

\begin{proof}
The unit space of $\gs$ is $\gs^{(0)}\simeq\Omega$. The map $\nu:\Omega\to\Rr(\gs)$ defined by $D\mapsto\nu^D(f):=\int_G f(D,h) \, dm_L(h)$ is a transverse function of $\gs$. Thus, \cite[Lemma~2.9~(b)]{LePeVe07} yields that $\int_\Omega \beta^D(F)\, d\mu(D)$ is independent of $F$, if $F$ is a measurable, non-negative function and $\nu\ast F = 1$ on $\Xx$, where
$$
(\nu\ast F )(D,x)
	= \int_{\gs^{\pi(D,x)}} 
		F\big(J\big(g^{-1}D,g^{-1}\big)(D,x)\big)
		\, d\nu^{\pi(D,x)}(D,g)
	= \int_G F(g^{-1}D,g^{-1}x) \, dm_L(g)
$$
for $(D,x)\in\Xx$. 
\end{proof}

\medskip

Next, we show that $L^2(\Xx,\mu\circ\alpha)$ fibers naturally over $\Omega$ as an direct integral invoking the $G$-invariance of the window function. Specifically, consider the vector space $S\subseteq\prod_{D\in\Omega}\ell^2(D,\win)$ of all measurable functions $u:\Xx\to\CC$ satisfying $u(D,\cdot)\in\ell^2(D,\win)$. With this measurable structure $S$ at hand, the family $\big(\ell^2(D,\win)\big)_{D\in\Omega}$ defines a measurable field of Hilbert spaces in terms of \cite[Part~II, Chapter~I.3, Definition~1]{Dixmier81}. 

\begin{proposition}
\label{prop:DirInt}
Let $\mu$ be a $G$-invariant probability measure on $\Omega$. Then the map 
$$
U:L^2(\Xx,\mu\circ\alpha)\to \int_\Omega^\oplus \ell^2(D,\win)\, d\mu(D)
	,\quad
	U(f)(D)(x):= f((D,x)),
$$
is unitary.
\end{proposition}

\begin{proof}
According to \cite[Lemma~2.9]{LeSt03-Algebras}, the Hilbert spaces $\int_\Omega^\oplus \ell^2(\Xx^D,\alpha^D)\, d\mu(D)$ and $L^2(\Xx,\mu\circ\alpha)$ are unitarily equivalent. Since $\Xx^D\simeq \{x\in D\,|\,\win_D(x)\neq 0\}=:\tilde{D}$, we derive $\ell^2(\tilde{D},\win)\simeq\ell^2(\Xx^D,\alpha^D)$ by the choice of the random variable $\alpha$. Recall that we chose the convention to denote $\ell^2(\tilde{D},\win)$ by $\ell^2(D,\win)$ in Section~\ref{ssec:RandOp}.
\end{proof}

\medskip

A family of operators $A:=(A_D)_{D\in\Omega}$ is called {\em measurable} if $D\mapsto\langle u_D, A_D v_D\rangle_{D,\win}$ is measurable for every $u,v\in S$. Furthermore, $A$ is {\em bounded} if $\sup_{D\in\Omega}\|A_D\|<\infty$ and $A$ is {\em equivariant} if $A_{gD}=T_gA_DT_g^*$ where $T_g:\ell^2(D,\win)\to\ell^2(gD,\win),$ $T_g u := u(g^{-1}\cdot)$. A measurable, bounded, equivariant family of operators $A:=(A_D)_{D\in\Omega}$ is called {\em random bounded operator}.

\begin{definition}
\label{def:vNalg}
Let $G$ and $H$ be lcsc groups where $G$ is unimodular, $\win\in\Cc_c(H)$ be a continuous window function and $\mu$ be a $(G\times H)$-invariant probability measure on $\Omega\in\IDelGH$. Then the associated von-Neumann algebra is defined by
$$
\Nn(\Omega,G\times H,\mu,\win) \; := \;
	\{
		A=(A_D)_{D\in\Omega}
		\;|\;
		A \text{ measurable, bounded, equivariant} 
	\}/_\sim,
$$
where two random bounded operators $A$ and $B$ are equivalent if $A_D=B_D$ $\mu$-almost everywhere.
\end{definition}

\medskip

{\bf Proof of Proposition~\ref{prop:GenTrace}.}
For every random bounded operator $A$, $\beta_A:\Omega\to \Rr(\Xx), \beta_A^D(f):= tr(A_D M_f(D))$ is a random variable, where $M_f$ is the multiplication operator by $f$. Thus, $\tau$ is independent of $\rho$ by Lemma~\ref{lem:RanVarInd}. That $\tau$ is a faithful, semifinite weight follows by \cite{LePeVe07}. The finiteness of $\tau$ follows as $\mu$ is a probability measure. It is left to show that $\tau$ is a trace: the main strategy is to show that every element in the von Neumann algebra is a Carleman operator, see e.g. \cite{LePeVe07} and \cite[Proposition~3.6]{LePoSc16}. According to \cite[Proposition~4]{LePeVe07}, the Carleman operators form a right ideal in $\Nn(\Omega,G\times H,\mu,\win)$. Using Urysohn's Lemma there is a $\psi\in\Cc_c(G\times H)$ such that $\supp(\psi)\subseteq U$ and $\psi(e)=1$, where $U\subseteq G$ open is the parameter of the Delone sets $\Delr$. The identity operator $I(x,D,y):=\psi(x^{-1}y)\in\Kfin$ is a Carleman operator. Thus, Proposition~\ref{prop:GenTrace} is a consequence of \cite[Theorem~4.2]{LePeVe07}.
\hfill$\Box$

\bibliographystyle{amsalphasort}
\bibliography{references}
\end{document}